\documentclass[aos]{imsart}
\usepackage{setup}
\usepackage{newcommands}
\usepackage{xr}

\begin{document}
	\begin{frontmatter}
		\title{Adaptive Estimation of Multivariate Piecewise Polynomials and Bounded Variation Functions by Optimal Decision Trees}              
		\runtitle{Adaptive estimation by optimal decision trees}  
		\runtitle{Adaptive estimation by optimal decision trees\;\;\;\;\;\;\;}

		\begin{aug}
			\author{\fnms{Sabyasachi} \snm{Chatterjee}\ead[label=e1]{sc1706@illinois.edu}}\footnote{{Supported by NSF Grant DMS-1916375}}
			\and
			\author{\fnms{Subhajit} \snm{Goswami}\ead[label=e2]{goswami@math.tifr.res.in}}\footnote{{Supported by an IDEX grant from Paris-Saclay and partially by a grant from the Infosys Foundation}}
			\blfootnote{Author names are sorted alphabetically.}
			
			\runauthor{Chatterjee, S. and Goswami S.}
			
			\affiliation{University of Illinois at Urbana-Champaign and Tata Institute of Fundamental Research 
			}

			\address{117 Illini Hall\\
	Champaign, IL 61820 \\
\href{mailto:sc1706@illinois.edu}{sc1706@illinois.edu}\\
	\phantom{as}}
\address{1, Homi Bhabha Road\\
	Colaba, Mumbai 400005, India\\
	\href{mailto:goswami@math.tifr.res.in}{goswami@math.tifr.res.in}\\
	\phantom{ac}}
			
		\end{aug}

		\begin{abstract}
			Proposed by \cite{donoho1997cart}, Dyadic CART is a nonparametric regression 
			method which computes a globally optimal dyadic decision tree and fits piecewise constant 
			functions in two dimensions. 
			In this article we 
			define and study Dyadic CART and a closely related estimator, namely Optimal Regression Tree (ORT), in the context of estimating piecewise smooth functions in 
			general dimensions in the fixed design setup. More precisely, these optimal decision tree estimators fit piecewise polynomials of any given degree. 
			Like Dyadic CART in two dimensions, we reason that 
			these estimators can also be computed in polynomial 
			time in the sample size $N$ via dynamic programming. 
			We prove oracle inequalities for the finite 
			sample risk of Dyadic CART and ORT which imply tight {risk} bounds for several function 
			classes of interest. Firstly, they imply that the 
			finite sample risk of ORT of {\em order} $r \geq 0$ is 
			always bounded by $C k \frac{\log N}{N}$ 
			whenever the regression function is 
			piecewise polynomial of degree $r$ on some reasonably regular axis aligned 
			rectangular partition of the domain with at most $k$ 
			rectangles. 
			Beyond the univariate case, such guarantees are 
			scarcely available in the literature for computationally efficient estimators. Secondly, our 
			oracle inequalities uncover minimax rate optimality and adaptivity of the Dyadic CART estimator for function spaces with 
			bounded variation. We consider two function spaces of recent interest where multivariate total variation 
			denoising and univariate trend filtering are the state 
			of the art methods. We show that Dyadic CART enjoys certain advantages over these estimators while still 
			maintaining all their known guarantees.
		\end{abstract}

	\end{frontmatter}
\section{Introduction}\label{Sec:intro}
Decision Trees are a widely used technique for nonparametric regression and classification. Decision Trees result in interpretable models and form a building block for more complicated methods such as bagging, boosting and random forests. See~\cite{loh2014fifty} and references therein for a detailed review. The most prominent example of decision trees is classification and regression trees (CART), proposed by~\cite{breiman1984classification}. CART operates in two stages. In the first stage, it recursively partitions the 
space of predictor variables in a greedy top down fashion. Starting from the root node, a locally optimal split is determined by an appropriate optimization criterion and then the process is iterated 
for each of the resulting child nodes. The final partition or decision tree is reached when a stopping criterion is met for each resulting node. In the second stage, the final tree is pruned by 
what is called \textit{cost complexity pruning}  where the cost of a pruned tree thus obtained is proportional to the number of the leaves of the tree; see
\:Section $9.2$ in~\cite{friedman2001elements} for details. 

A possible shortcoming of CART is that it produces locally optimal decision trees. It is natural to attempt to resolve this by computing a globally optimal decision tree. However, computing globally optimal decision tree is computationally a hard problem. It is known (see~\cite{laurent1976constructing}) that computing an optimal (in a particular sense) binary tree is NP hard. A recent paper of~\cite{bertsimas2017optimal} sets up an 
optimization problem (see in~\cite[Equation~$1$]{bertsimas2017optimal}) in the context of classification, which aims to minimize (among all decision trees) misclassification error of a tree plus a penalty proportional to its number of leaves. The paper formulates this problem as an instance of mixed integer optimization (MIO) and claims 
that modern MIO developments allow for solving reasonably sized problems. It then demonstrates extensive experiments for simulated and real data sets where the optimal tree outperforms the usual CART. These experiments seem to provide strong empirical evidence that optimal decision trees, if computed, can perform significantly better than CART. Another shortcoming of CART is that it is typically very hard to theoretically analyze the full algorithm because of the sequence of data dependent splits. Some results (related to the current paper) exist for the subtree obtained in the pruning stage, conditional on the maximal tree grown in the first stage; see~\cite{gey2005model} and references therein. Theoretical guarantees for the widely used Random Forests are also typically hard to obtain inspite of much recent work see~\cite{scornet2015consistency},~\cite{wager2015adaptive},~\cite{ishwaran2015effect} and references therein. On the other hand, theoretical analysis for optimal decision trees can be obtained since it can be seen as penalized empirical risk minimization.

One class of decision trees for which an optimal tree can be computed efficiently, in low to moderate dimensions, is the class of \textit{dyadic decision trees}. These trees are constructed from recursive dyadic partitioning. In the case of regression on a two-dimensional grid design, the paper~\cite{donoho1997cart} proposed a penalized least squares estimator called the Dyadic CART estimator. The author showed that it is possible to compute this estimator by a fast bottom up dynamic program which has linear time computational complexity $O(n \times n)$ for a $n \times n$ grid. Moreover, the author showed that Dyadic CART satisfies an oracle risk bound which in turn was used to show that it is adaptively minimax rate optimal over classes of anisotropically smooth bivariate functions. Ideas in this paper were later used in~\cite{nowak2004estimating} in the context of adaptively estimating piecewise Holder smooth functions. The idea of dyadic partitioning were also used in classification in papers such as~\cite{scott2006minimax}
and~\cite{blanchard2007optimal} who studied penalized empirical risk minimization over dyadic decision trees of a fixed maximal depth. They also proved oracle risk bounds and showed minimax rate optimality for appropriate classes of classifiers. Minimax rates of convergence have also been obtained for various models of dyadic classification trees in~\cite{lecue2008classification}. In the related problem of density estimation, dyadic partitioning estimators have also been studied in the context of estimating piecewise polynomial densities; see~\cite{willett2007multiscale}. This current paper focusses on the regression setting and follows this line of work of studying optimal decision trees, proving an oracle risk bound and then investigating implications for certain function classes of interest. The optimal decision trees we study in this paper are computable in time polynomial in the sample size. 

In particular, in this paper, we study two decision tree methods for estimating regression functions in general dimensions in the context of estimating some nonsmooth function classes of recent interest. We focus on the fixed lattice design case like in~\cite{donoho1997cart}. 
The first method is an optimal dyadic regression tree and is exactly the same as Dyadic CART in~\cite{donoho1997cart} when the dimension is $2$. The second method is an Optimal Regression Tree (ORT), very much in the sense of~\cite{bertsimas2017optimal}, applied to fixed lattice design regression. Here the estimator is computed by optimizing a penalized least squares criterion over the set of all --- not just dyadic --- decision trees. We make the crucial observation that this estimator can be computed by a dynamic programming approach when the design points fall on a lattice. Thus, for instance, one does not need to resort to mixed integer programming and this dynamic program has computational complexity polynomial in the sample size. This observation may be known to the experts but we are unaware of an exact reference. Like in~\cite{donoho1997cart} we show it is possible to prove an oracle risk bound (see Theorem~\ref{thm:adapt}) for both of our optimal decision tree estimators. We then apply this oracle risk bound to three function classes of recent interest by employing approximation theoretic inequalities and show that these optimal decision trees have excellent adaptive and worst case performance.

Overall in this paper, we revisit the classical idea of recursive partitioning in the context of finding answers to several
unsolved questions about some classes of functions of recent interest in the nonparametric regression 
literature. In the course of doing so, we have come up with as well as brought forward several interesting ideas from different areas relevant for the study of regression trees such as dynamic programming, computational geometry and discrete Sobolev type 
inequalities for vector / matrix approximation. We believe that the main novel 
aspect of the current work is to recognize, prove and point out --- by an amalgamation of 
these ideas --- that optimal regression trees often provide a better alternative to the state of 
the art convex optimization methods in the sense they are simultaneously 
(near-) minimax rate optimal, adaptive to the complexity of the underlying signal (under fewer assumptions) and 
computationally more efficient for some classes of functions of recent interest. To the best of our knowledge, our paper is the first one among a series of recent works that shows the efficacy of computationally efficient optimal regression tree estimators in these particular nonparametric regression problems.
We now describe the function classes we consider in this paper and briefly outline our results and contributions. 

\begin{itemize}
	\item \textbf{Piecewise Polynomial Functions}: We address the problem of estimating multivariate functions that are (or close to) \textit{piecewise polynomial} of some fixed degree on some unknown partition of the domain into axis aligned rectangles. This includes function classes such as piecewise constant/linear/quadratic etc. on axis aligned rectangles. An oracle, who knows the true rectangular partition, i.e the number of axis aligned rectangles and their arrangement, can just perform least squares separately for data falling within each rectangle. This oracle estimator provides a benchmark for adaptively optimal performance. The main question of interest to us is how to construct an estimator which is efficiently computable and attains risk as close as possible to the risk of this oracle estimator. 
	To the best of our knowledge, this question has not been answered in multivariate settings. In this paper, we propose that our optimal regression tree (ORT) estimator solves this question to a 
	considerable extent. Section~\ref{secmpp} describes all of our results under this topic. It is worthwhile to mention here that we 
	also focus on cases where the true rectangular partition does not correspond to any decision tree (see Figure~\ref{fig3}) which necessarily has a hierarchical structure. We call such partitions 
	nonhierarchical. Even for such nonhierarchical partitions, we make the case that ORT continues to perform well (see our results in Section~\ref{sec:non}). We are not aware of nonhierarchical 
	partitions being studied before in the literature. Here our proof technique uses results from computational geometry which relate the size of any given (possibly nonhierarchical) rectangular partition 
	to that of the minimal hierarchical partition refining it. 
	
	
	\smallskip

	\item \textbf{Multivariate Bounded Variation Functions}: Consider the function class whose total variation (defined later in Section~\ref{sec:tv}) is bounded by some number. This is a classical function class for nonparametric regression since it contains functions which demonstrate spatially heterogenous smoothness; see Section $6.2$ in~\cite{tibshiraninonparametric} and references therein. Perhaps, the most natural estimator for this class of functions is what is called the Total Variation Denoising (TVD) estimator. The two dimensional version of this estimator is also very popularly used for image denoising; see~\cite{rudin1992nonlinear}. It is known that a well tuned TVD estimator is minimax rate optimal for this class in all dimensions; see~\cite{hutter2016optimal} and~\cite{sadhanala2016total}. Also, in the univariate case, it is known that the TVD estimator adapts to piecewise constant functions and attains a near oracle risk with parametric rate of convergence; see~\cite{guntuboyina2020adaptive} and references therein. However, even in two dimensions, the TVD estimator provably cannot attain the near parametric rate of convergence for piecewise constant truths. This is a result (Theorem $2.3$) in a previous article by the same authors~\cite{chatterjee2019new}.

	It would be desirable for an estimator to attain the minimax rate among bounded variation functions as well as retain the near parametric rate of convergence for piecewise constant truths in multivariate settings. Our contribution here is to establish that Dyadic CART enjoys these two desired properties in all dimensions.  
	We also show that the Dyadic CART adapts to the intrinsic dimensionality of the function in a particular sense. Therorem~\ref{thm:dcadap} is our main result under this topic. Our proof technique for Theorem~\ref{thm:dcadap} involves a recursive partitioning strategy to approximate any given bounded variation function by a piecewise constant function (see Proposition~\ref{prop:division}). We prove an inequality which can be thought of as the discrete version of the classical Gagliardo-Sobolev-Nirenberg inequality (see Proposition~\ref{prop:gagliardo}) which plays a key role in the proof.

	As far as we are aware, Dyadic CART has not been investigated before in the context of estimating bounded variation functions. Coupled with the fact that Dyadic CART can be computed in time linear in the sample size, our results put forth the Dyadic CART estimator as a fast and viable option for estimating bounded variation functions. 
	
	\smallskip
	
	\item \textbf{Univariate Bounded Variation Functions of higher order}: 
	Higher order versions of the space of bounded variation functions has also been considered in nonparametric regression, albeit mostly in the univariate case. One can consider the univariate function 
	class of all $r$ times (weakly) differentiable functions, whose $r$ th derivative is of bounded variation. A seminal result of~\cite{donoho1998minimax} shows that a wavelet threshholding estimator attains the minimax rate in this problem. Locally adaptive regression splines, proposed by~\cite{mammen1997locally}, is also known to achieve the minimax rate in this problem. 
	Recently, Trend Filtering, proposed by~\cite{kim2009ell_1}, has proved to be a popular nonparametric regression method. Trend Filtering is very closely related to locally adaptive regression splines and is also minimax rate optimal over the space of higher order bounded variation functions; see~\cite{tibshirani2014adaptive} and references therein. Moreover, it is known that Trend Filtering adapts to functions which are piecewise polynomials with regularity at the knots. If the number of pieces is not too large and the length of the pieces is not too small, a well tuned Trend Filtering estimator can attain near parametric risk as shown in~\cite{guntuboyina2020adaptive}.

	Our main contribution here is to show that the univariate Dyadic CART estimator is also minimax rate optimal in this problem and enjoys near parametric rate of convergence for piecewise polynomials; see Theorem~\ref{thm:slowrate} and 
	Theorem~\ref{thm:fastrate}. Moreover, we show that Dyadic CART requires less regularity assumptions on the true function than what Trend Filtering requires for the near parametric rate of convergence to hold. Theorem~\ref{thm:fastrate} follows directly from a combination of our oracle risk bound and a result about refining an arbitrary (possibly non dyadic) univariate partition to a dyadic one (see Lemma~\ref{lem:1dpartbd}). Our proof technique for Theorem~\ref{thm:slowrate} again involves a recursive partitioning strategy to approximate any given higher order bounded variation function by a piecewise polynomial function (see Proposition~\ref{prop:piecewise}). We prove an inequality (see Lemma~\ref{lem:approxpoly}) quantifying the error of approximating a higher order bounded variation function by a single polynomial which plays a key role in the proof.

	Again, as far as we are aware, Dyadic CART has not been investigated before in the context of estimating univariate higher order bounded 
	variation functions. Coupled with the fact that Dyadic CART is computable in time 
	linear in the sample size, our results again provide a fast and viable alternative for estimating univariate higher order bounded variation functions. 

	
\end{itemize}

The oracle risk bound in Theorem~\ref{thm:adapt} which holds for the optimal decision trees studied in this paper may imply near optimal results for other function classes as well. In Section~\ref{Sec:discuss}, we mention some consequences of our oracle risk bounds for shape constrained function classes. We then describe a version of our estimators which can be implemented for arbitrary data with random design and also discuss an extension of our results for dependent noise.

\subsection{Problem Setting and Definitions}

Let us denote the $d$ dimensional lattice with $N$ points by $L_{d,n} \coloneqq \{1,\dots,n\}^d$ where $N = n^{d}.$ Throughout this paper we will consider the standard fixed design setting where we treat the $N$ design points as fixed and located on the $d$ dimensional grid/lattice $L_{d,n}.$ One may think of the design points embedded in $[0,1]^d$ and of the form $\frac{1}{n}(i_1,\dots,i_d)$ where $(i_1,\dots,i_d) \in L_{d,n}$. This lattice design is quite commonly used for theoretical studies in multidimensional nonparametric function estimation (see, e.g.~\cite{nemirovski2000topics}). The lattice design is also the natural setting for certain applications such as image denoising, matrix/tensor estimation. All our results will be for the lattice design setting. In Section~\ref{Sec:discuss}, we make some observations and comments about possible extensions to the random design case. 

Letting $\theta^*$ denote the evaluation on the grid of the underlying regression function $f$, our observation 
model becomes $y = \theta^* + \sigma Z$ where 
$y,\theta^*,Z$ are real valued functions on $L_{d,n}$ and 
hence are $d$ dimensional arrays. Furthermore, $Z$ is a noise array consisting of independent standard Gaussian entries and 
$\sigma > 0$ is an unknown standard deviation of the noise 
entries. For an estimator $\hat{\theta}$, we will evaluate 
its performance by the usual fixed design expected mean 
squared error $$\MSE(\hat{\theta},\theta^*) \coloneqq 
\frac{1}{N}\:\E_{\theta^*} \|\hat{\theta} - \theta^*\|^2.$$ 
Here $\|.\|$ refers to the usual Euclidean norm of an array 
where we treat an array as a vector in $\R^N.$

Let us define the interval of positive integers $[a,b] = \{i \in \Z_{+}: a \leq i \leq b\}$ where $\Z_{+}$ denotes 
the set of positive integers. For a positive integer $n$ we 
also denote the set $[1,n]$ by just $[n].$ A subset $R 
\subset L_{d,n}$ is called an \textit{axis aligned 
	rectangle} if $R$ is a product of 
intervals, i.e. $R = \prod_{i = 1}^{d} [a_i,b_i].$ 
Henceforth, we will just use the word rectangle to denote an 
axis aligned rectangle. Let us define a \textit{rectangular 
	partition} of $L_{d,n}$ to be a set of rectangles 
$\mathcal{R}$ such that (a) the rectangles in $\mathcal{R}$ 
are pairwise disjoint and (b) $\cup_{R \in \mathcal{R}} R = 
L_{d,n}.$

Recall that a multivariate polynomial of degree at most $r \geq 0$ is a finite linear combination of the monomials $\Pi_{i = 1}^{d} (x_i)^{r_i}$ satisfying $\sum_{i = 1}^{d} r_i 
\leq r.$ It is thus clear that they form a linear space of dimension 
$K_{r,d} \coloneqq {r + d - 1 \choose d - 1}.$ Let us now define the set of discrete multivariate polynomial arrays as 
\begin{align*}\label{eq:polydef}
	\mathcal{F}^{(r)}_{d,n} = \big\{\theta \in \R^{L_{d,n}}: \theta(i_1/n,\dots,i_d/n) = &f(i_1/n,\dots,i_d/n)\:\:\forall (i_1,\dots,i_d) \in [n]^d \\&
	\text{for some polynomial $f$ of degree at most $r$}\big\}.
\end{align*}

For a given rectangle $R \subset L_{d,n}$ and any $\theta \in \R^{L_{d,n}}$ let us denote the array obtained by restricting $\theta$ to $R$ by $\theta_{R}.$ We say that $\theta$ is a degree $r$ polynomial on the rectangle $R$ if $\theta_{R} = \alpha_{R}$ for some $\alpha \in \mathcal{F}^{(r)}_{d,n}.$ 

For a given array $\theta \in \R^{L_{d,n}}$, let \textit{$k^{(r)}(\theta)$ denote the smallest positive integer $k$ such that a set of $k$ rectangles $R_1,\dots,R_k$ form a rectangular partition of $L_{d,n}$ and the restricted array $\theta_{R_i}$ is a degree $r$ polynomial for all $1 \leq i \leq k.$} 
In other words, $k^{(r)}(\theta)$ is the cardinality of the minimal rectangular partition of $L_{d,n}$ such that $\theta$ is piecewise polynomial of degree $r$ on the partition.

\subsection{Description of Estimators}
The estimators we consider in this manuscript compute a data dependent decision tree (which is globally optimal in a certain sense) and then fit polynomials within each cell/rectangle of the decision tree. As mentioned before, computing decision trees greedily and then fitting a constant value within each cell of the decision tree has a long history and is what the usual CART does. Fitting polynomials on such greedily grown decision trees is a natural extension of CART and has also been proposed in the literature; see~\cite{chaudhuri1994piecewise}. The main difference between these estimators and our estimators is that our decision trees are computed as a global optimizer over the set of all decision trees. In particular, they are not grown greedily and there is no stopping rule that is required. The ideas here are mainly inspired by~\cite{donoho1997cart}. We now define our estimators precisely.


Recall the definition of $k^{(r)}(\theta).$ A natural estimator which fits piecewise polynomial functions of degree $r \geq 0$ on axis aligned rectangles is the following fully penalized LSE of order $r$:
\begin{equation*}
	\hat{\theta}^{(r)}_{\all,\lambda} \coloneqq \argmin_{\theta \in \R^{L_{d,n}}} \big(\|y - \theta\|^2 + \lambda k^{(r)}(\theta)\big).
\end{equation*}


Let us denote the set of all rectangular partitions of $L_{d,n}$ as $\mathcal{P}_{\all, d, n}.$ For each rectangular partition $\Pi \in \mathcal{P}_{\all, d, n}$ and each nonnegative integer $r$, let the (linear) subspace $S^{(r)}(\Pi)$ 
comprise all arrays which are degree $r$ polynomial on each of the rectangles constituting $\Pi.$ 
For a generic subspace $S \subset \R^N$ let us denote its dimension by $Dim(S)$ and the associated orthogonal projection matrix by $O_{S}.$ Clearly the dimension of the subspace $S^{(r)}(\Pi)$ is $K_{r,d} |\Pi|$ where $|\Pi|$ is the cardinality of the partition. Now note that we can also write $\hat{\theta}^{(r)}_{\all,\lambda} = O_{S^{(r)}(\hat{\Pi}(\lambda))} y$ where $\hat{\Pi}(\lambda)$ is a data dependent partition defined as 
\begin{equation}\label{eq:defplse}
\hat{\Pi}(\lambda) = \argmin_{\Pi: \Pi \in \mathcal{P}_{\all,d,n}}\big( \|y - O_{S^{(r)}(\Pi)} y\|^2 + \lambda |\Pi| \big).
\end{equation}

Thus, computing $\hat{\theta}^{(r)}_{\lambda,\all}$ really involves optimizing over all 
rectangular partitions $\Pi \in \mathcal{P}_{\all,d,n}.$ Therefore, one may anticipate that 
the major roadblock in using this estimator would be computation. For any fixed $d$, the cardinality of $\mathcal{P}_{\all, d, n}$ is at least stretched-exponential in $N.$ Thus, a brute 
force method is infeasible. However, for $d = 1$, a rectangular partition is a set of 
contiguous blocks of intervals which has enough structure so that a dynamic programming 
approach is amenable. The set of all multivariate rectangular partitions is a more 
complicated object and the corresponding computation is likely to be provably hard. This is 
where the idea of~\cite{donoho1997cart} comes in who considers the Dyadic CART estimator (for 
$r = 0$ and $d = 2$) for fitting piecewise constant functions. As we will now explain, it 
turns out that if we constrain the optimization in~\eqref{eq:defplse} to optimize over special 
subclasses of rectangular partitions of $L_{d,n}$, a dynamic programming approach again becomes tractable. The Dyadic CART estimator is one such constrained version of the optimization problem in~\eqref{eq:defplse}. We now precisely define these subclasses of rectangular partitions.


\subsubsection{Description of Dyadic CART of order $r \geq 0$}
Let us consider a generic discrete interval $[a,b].$ We define a \textit{dyadic split} of the interval to be a split of the interval $[a,b]$ into two equal intervals. To be concrete, the interval $[a,b]$ is split into the intervals $[a,a - 1 + \ceil{(b - a + 1)/2}]$ and $[a + 
\ceil{(b - a + 1)/2}, b].$ Now consider a generic rectangle $R = \prod_{i = 1}^{d} [a_i,b_i].$ A \textit{dyadic split} of the rectangle $R$ involves the choice of a coordinate $1 \leq j \leq d$ to be split and then the $j$-th interval in the product defining the rectangle $R$ undergoes a dyadic split. Thus, a dyadic split of $R$ produces two sub rectangles $R_1$ and $R_2$ where $R_2 = R \cap R_1^{c}$ and $R_1$ is of the following form for some $j \in [d]$,
\begin{equation*}
	R_1 = \prod_{i = 1}^{j - 1} [a_i,b_i] \times [a_j ,a_j - 1 + \ceil{(b_j - a_j + 1) / 2}] \times \prod_{i = j + 1}^{d} [a_i,b_i].
\end{equation*}

Starting from the trivial partition which is just $L_{d,n}$ itself, 
we can create a refined partition by dyadically splitting $L_{d,n}.$  This will result in a partition of $L_{d,n}$ into two rectangles. We can now keep on dividing recursively, generating new partitions. In general, if at some stage we have the partition $\Pi = (R_1,\dots,R_k)$, we can choose any of the rectangles $R_i$ and dyadically split it to get a refinement of $\Pi$ with $k + 1$ nonempty rectangles. \textit{A recursive dyadic partition} (RDP) is any partition reachable by such successive dyadic splitting. Let us denote the set of all recursive dyadic partitions of $L_{d,n}$ as $\mathcal{P}_{\rdp,d,n}.$ Indeed, a natural way of encoding any RDP of $L_{d,n}$ is by a binary tree where each nonleaf node is labeled by an integer in $[d].$ This labeling corresponds to the choice of the coordinate that was used for the split. 


We can now consider a constrained version of $\hat{\theta}^{(r)}_{\all,\lambda}$ which only optimizes over $\mathcal{P}_{\rdp,d,n}$ instead of optimizing over $\mathcal{P}_{\all,d,n}.$ Let us define $\hat{\theta}^{(r)}_{\rdp,\lambda} = O_{S^{(r)}(\hat{\Pi}_{\rdp}(\lambda))} y$ where $\hat{\Pi}_{\rdp}(\lambda)$ is a data dependent partition defined as 
\begin{equation*}
	\hat{\Pi}_{\rdp}(\lambda) = \argmin_{\Pi: \Pi \in \mathcal{P}_{\rdp,d,n}} \big(\|y - O_{S^{(r)}(\Pi)}\|^2 + \lambda  |\Pi|\big). 
\end{equation*}


The estimator $\hat{\theta}^{(r)}_{\rdp,\lambda}$ is precisely the Dyadic CART estimator which was proposed in~\cite{donoho1997cart} in the case when $d = 2$ and $r = 0.$ The author studied this estimator for estimating anisotropic smooth functions of two variables which exhibit different degrees of smoothness in the two variables. However, to the best of our knowledge, the risk properties of the Dyadic CART estimator (for $r = 0$) has not been examined in the context of estimating nonsmooth function classes such as piecewise constant and bounded variation functions. For $r \geq 1,$ the above estimator appears to not have been proposed and studied in the literature before. We call the estimator $\hat{\theta}^{(r)}_{\rdp,\lambda}$ as {\em Dyadic CART of order $r$}.




\subsubsection{Description of ORT of order $r \geq 0$}
For our purposes, we would need to consider a larger class of partitions than $\mathcal{P}_{\rdp,d,n}.$ To generate a RDP, for each rectangle we choose a dimension to split and then split at the midpoint. Instead of splitting at the midpoint, it is natural to allow the split to be at an arbitrary position. To that end, we define a \textit{hierarchical split} of the interval to be a split of the interval $[a,b]$ into two intervals, but not necessarily equal sized. To be concrete, the interval $[a,b]$ is split into the intervals $[a,\ell]$ and $[\ell + 1, b]$ for some $a \leq \ell \leq b.$ Now consider a generic rectangle $R = \prod_{i = 1}^{d} [a_i,b_i].$ A \textit{hierarchical split} of the rectangle $R$ involves the choice of a coordinate $1 \leq j \leq d$ to be split and then the $j$-th 
interval in the product defining the rectangle $R$ undergoes a hierarchical split. Thus, a hierarchical split of $R$ produces two sub rectangles $R_1$ and $R_2$ where $R_2 = R \cap R_1^{c}$ and $R_1$ is of the following form for some $1 \leq j \leq d$ and $a_j \leq \ell \leq b_j$,
\begin{equation*}
	R_1 = \prod_{i = 1}^{j - 1} [a_i,b_i] \times [a_j,\ell] \times \prod_{i = j + 1}^{d} [a_i,b_i].
\end{equation*}

Starting from the trivial partition $L_{d,n}$ itself, we can now generate partitions by splitting $L_{d,n}$ hierarchically. Again, in general if at some stage we obtain the partition $\Pi = (R_1,\dots,R_k)$, we can choose any of the rectangles $R_i$ and split it hierarchically to obtain $k + 1$ nonempty rectangles now. \textit{A hierarchical partition} is any partition reachable by such hierarchical splits. We denote the set of all hierarchical partitions of $L_{d,n}$ as $\mathcal{P}_{\hier, d, n}.$ Note that a hierarchical partition is in one to one correspondence with decision trees and thus, $\mathcal{P}_{\hier, d, n}$ can be thought of as the set of all decision trees.

Clearly, 
$$\mathcal{P}_{\rdp,d,n} \subset \mathcal{P}_{\hier,d,n} \subset \mathcal{P}_{\all, d, n}.$$ In fact, the inclusions are strict as shown in Figure~\ref{fig1}. In particular, there exist partitions which are not hierarchical. 
\begin{figure}\label{fig3}
	\begin{center}
		\includegraphics[scale=0.5]{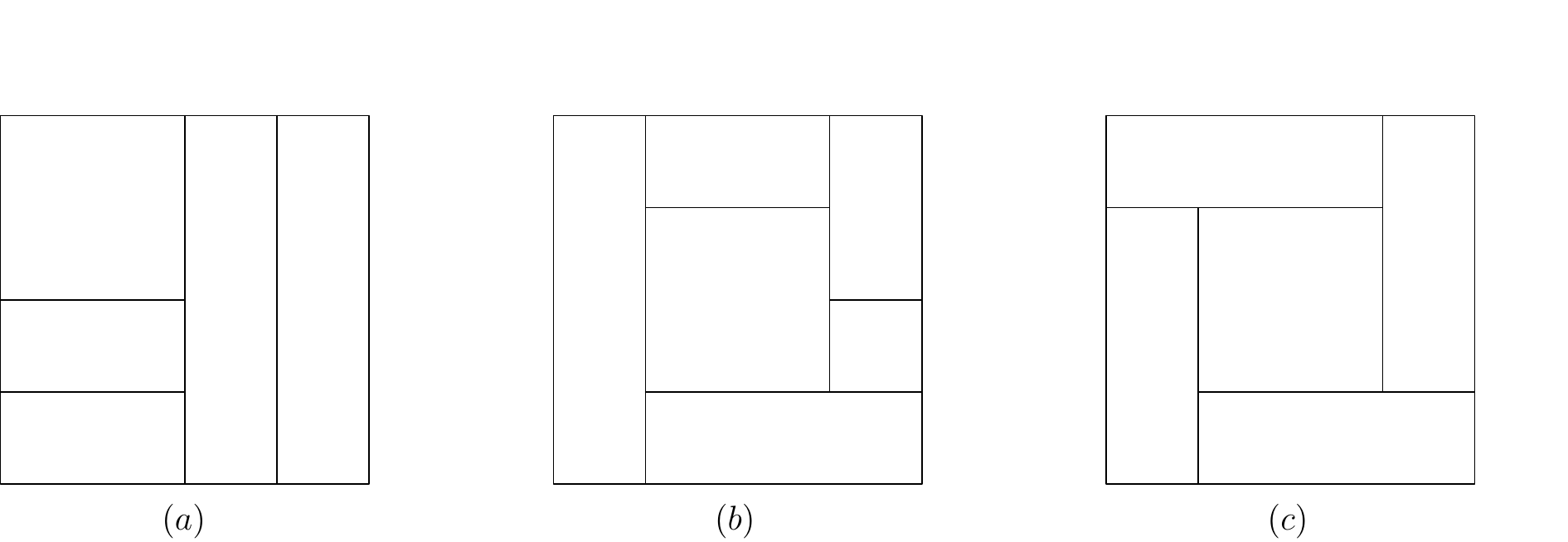} 
	\end{center}
	\caption{\textit{Figure~(a) is an example of a recursive dyadic partition of the square. Figure~(b) is nondyadic but is a hierarchical partition. Figure~(c) is an example of a nonhierarchical partition. An easy way to see this is that there is no split from top to bottom or left to right.}} 
	\label{fig1}
\end{figure}


We can now consider another constrained version of $\hat{\theta}^{(r)}_{\all,\lambda}$ which optimizes only over $\mathcal{P}_{\hier,d,n}.$ Let us define $\hat{\theta}^{(r)}_{\hier, \lambda} = O_{S^{(r)}(\hat{\Pi}_{\hier}(\lambda))} y$ where $\hat{\Pi}_{\hier}(\lambda)$ is a data dependent partition defined as 
\begin{equation*}
	\hat{\Pi}_{\hier}(\lambda) = \argmin_{\Pi: \Pi \in \mathcal{P}_{\hier,d,n}} \big(\|y - O_{S^{(r)}(\Pi)} y\|^2 + \lambda |\Pi|\big). 
\end{equation*}


Although this is a natural extension of Dyadic CART, we are unable to pinpoint an exact reference where this estimator has been explicitly proposed or studied in the statistics literature. The above optimization problem is an analog of the optimal decision tree problem laid out in~\cite{bertsimas2017optimal}. The difference is that~\cite{bertsimas2017optimal} is considering classification whereas we are considering fixed lattice design regression. Note that the above optimization problem is different from the usual pruning of a tree that is done at the second stage of CART. Pruning can only result in subtrees of the full tree obtained in the first stage whereas the above optimization is over all rectangular partitions $ \Pi \in \mathcal{P}_{\hier,d,n}.$ We name the estimator $\hat{\theta}^{(r)}_{\lambda, \hier}$ as \textit{Optimal Regression Tree} (ORT) of order $r$.



\subsection{Both Dyadic CART and ORT of all orders are efficiently computable}
The crucial fact about $\hat{\theta}^{(r)}_{\rdp,\lambda}$ and 
$\hat{\theta}^{(r)}_{\hier,\lambda}$ is that they can be computed 
efficiently and exactly using dynamic 
programming approaches. A dynamic program algorithm to 
compute $\hat{\Pi}_{\rdp,\lambda}$ for $d = 2$ and $r = 0$ was shown 
in~\cite{donoho1997cart}. This algorithm is extremely fast 
and can be computed in $O(N)$ (linear in sample size) time. 
The basic idea there can actually be extended to compute both Dyadic CART and ORT for any fixed $r,d$ with computational complexity given in the next lemma. The proof is given in Section~\ref{Sec:proofs} (in the supplementary file,). 

\begin{lemma}\label{lem:compu}
	There exists an absolute constant $C > 0$ such that the computational complexity, i.e. the number of elementary operations involved in the computation of ORT is bounded by:
	\begin{equation*}
		\begin{cases}
			CN^2\:nd, & \text{for} \:\:r = 0 \\
			CN^2\:(nd + d^{3r}) & \text{for} \:\:r \geq 1
		\end{cases}
	\end{equation*}
	Similarly, the computational complexity of Dyadic CART is bounded by:
	\begin{equation*}
		\begin{cases}
			C2^d\:Nd, & \text{for} \:\:r = 0 \\
			C2^d\:N\:d^{3r}& \text{for} \:\:r \geq 1.
		\end{cases}
	\end{equation*}
\end{lemma}


\begin{remark}
	{Since the proxy for the sample size $N \geq 2^d$ as soon as $n \geq 2$, it does not make sense to think of $d$ as large when reading the above computational complexity. The lattice design setting is really meaningful when $d$ is low to moderate and fixed and the number of samples per dimension $n$ is growing to $\infty$. Thus, one should look at the dependence of the computational complexity on $N$ and treat the factors depending on $d$ as constant}.
\end{remark}

\begin{remark}
	Even for $d = 1$, the brute force computation time is exponential in $N$ as the total number of hierarchical partitions is exponential in $N.$
\end{remark}

The rest of the paper is organized as follows. 
In Section~\ref{Sec:results} we state our oracle risk bound for 
Dyadic CART and ORT of all orders. Section~\ref{secmpp} describes 
applications of the oracle risk bound for ORT to multivariate 
piecewise polynomials. In Sections~\ref{sec:tv} and~\ref{Sec:uni} 
we state applications of the oracle risk bound for Dyadic CART to 
multivariate bounded variation functions in general dimensions and 
univariate bounded variation function classes of all orders 
respectively. In Section~\ref{sec:sim} we describe our simulation studies and in Section~\ref{Sec:discuss}, we summarize our 
results, reiterate the main contributions of this 
paper and discuss some matters related to our work here. Section~\ref{Sec:proofs} contains all the proofs of our results. In Section~\ref{Sec:appendix} we state and prove some auxiliary results that we use for proving our main results in the paper.

\textbf{Acknowledgements}: This research was supported by a NSF grant and an IDEX grant from Paris-Saclay. S.G.’s research was carried out in part as a member of the
Infosys-Chandrasekharan virtual center for Random Geometry, supported by a grant from the
Infosys Foundation. We thank the anonymous referees for their numerous helpful remarks and 
suggestions on an earlier manuscript of the paper. We also thank Adityanand 
Guntuboyina for many helpful comments. The project started when SG was a postdoctoral fellow 
at the Institut des Hautes \'{E}tudes Scientifiques (IHES).

\section{Oracle risk bounds for Dyadic CART and ORT}\label{Sec:results}
In this section we describe an oracle risk bound. We have to set up some notations and terminology first. Let $\mathcal{S}$ be any 
finite collection of subspaces of $\R^N.$ Recall that for a generic subspace $S\in \mathcal{S}$, we denote its dimension by $Dim(S).$  For any given $\theta \in \R^N$ let us define 
\begin{equation}\label{eq:compl}
k_{\mathcal{S}}(\theta) = \min\{Dim(S): S \in \mathcal{S}, \theta \in S\}
\end{equation}
where we adopt the convention that the infimum of an empty set is $\infty$.

For any $\theta \in \R^N,$ the number $k_{\mathcal{S}}(\theta)$ can be thought of as describing the complexity of $\theta$ with respect to the collection of subspaces $\mathcal{S}.$ Recall the definition of the nested classes of rectangular partitions $\mathcal{P}_{\rdp,d,n} \subset \mathcal{P}_{\hier,d,n} \subset \mathcal{P}_{\all,d,n}.$ Also recall that the subspace $S^{(r)}(\Pi)$ 
denotes all arrays which are degree $r$ polynomial on each of the rectangles constituting $\Pi.$ For any integer $r \geq 0$, these classes of partitions induce their respective collection of subspaces of $\R^N$ defined as follows:
\begin{equation*}
	\mathcal{S}^{(r)}_{a} = \{S^{(r)}(\Pi): \Pi \in \mathcal{P}_{a, d, n}\}
\end{equation*}
where $a \in \{\rdp,\hier,\all\}$. 
For any $\theta \in \R^{L_{d, n}}$ and any integer $r \geq 0$ let us observe its complexity with respect to the collection of subspaces $S^{(r)}_{a}$ is 
\begin{equation*}
	k_{\mathcal{S}^{(r)}_{a}}(\theta) = k^{(r)}_{a}(\theta)
\end{equation*}
where again $a \in \{\rdp,\hier,\all\}.$ Here $k^{(r)}_{\all}(\theta^*)$ is the same as $k^{(r)}(\theta^*)$ defined earlier and we use both notations interchangeably.

It is now clear that for any $\theta \in \R^N$ we have
\begin{equation}\label{eq:chain}
k_{\all}^{(r)}(\theta) \leq k_{\hier}^{(r)}(\theta) \leq k_{\rdp}^{(r)}(\theta).
\end{equation}


We are now ready to 
an oracle risk bound for all the three estimators $\hat{\theta}^{(r)}_{\all,\lambda},\hat{\theta}^{(r)}_{\rdp,\lambda}$ and $\hat{\theta}^{(r)}_{\hier,\lambda}.$ The theorem is proved in Section~\ref{Sec:proofs}.
\begin{theorem}\label{thm:adapt}
	Fix any integer $r \geq 0$ and recall that $K_{r,d} = {r + 
		d - 1 \choose d - 1}$ was defined earlier. There exists an absolute constant $C > 0$ such that for any $0 < \delta < 1$ if we set $\lambda \geq C K_{r,d} \frac{\sigma^2 \log N}{\delta}$, then we have the following risk bounds for $a \in \{\rdp,\hier,\all\}$,
	\begin{equation*}
		\E \|\hat{\theta}^{(r)}_{a,\lambda} - \theta^*\|^2 \leq \inf_{\theta \in \R^N} \big[\frac{(1 + \delta)}{(1 - \delta)}\:\|\theta - \theta^*\|^2 + \frac{\lambda}{1 - \delta}\:k^{(r)}_{a}(\theta)\big] + C \frac{\sigma^2}{\delta\:(1 - \delta)}
	\end{equation*}
\end{theorem}
Let us now discuss about some aspects of the above theorem.
\begin{itemize}
	\item Theorem~\ref{thm:adapt} gives an oracle risk bound for all the three estimators: the fully penalized LSE $\hat{\theta}^{(r)}_{\all,\lambda}$, the Dyadic CART estimator $\hat{\theta}^{(r)}_{\rdp,\lambda}$ and the ORT estimator $\hat{\theta}^{(r)}_{\hier,\lambda}$ of all orders $r \geq 0$ in all dimensions $d.$ An upper bound on the risks of these estimators is given by an oracle risk bound trading off squared error distance and complexity. Such a result already appeared in~\cite[Theorem~$7.2$]{donoho1997cart} for the Dyadic CART estimator in the special case when $r = 0$ and $d = 2$ with an additional multiplicative $\log N$ factor in front of the squared 
	distance term. Such an oracle risk bound also appeared in~\cite{nowak2004estimating} (see equation $8$) where an upper bound to the MSE in terms of approximation error plus estimation error is given. The main points we want to make by proving the above theorem are that firstly, it continues to hold for Dyadic CART and 
	ORT of all orders and in all dimensions. Secondly, this oracle inequality potentially implies tight bounds for several classes of functions of recent interest. 
	
	\smallskip
	
	\item Oracle risk bounds for Dyadic CART type estimators have also been shown for classification problems; see, e.g.~\cite{blanchard2007optimal} and~\cite{scott2006minimax}. 
	
	\smallskip
	
	\item Due to the inequality in~\eqref{eq:chain}, the risk bounds are ordered pointwise in $\theta^*.$ The risk bound is best for the 
	fully penalized LSE followed by ORT and then followed by Dyadic CART. 
	However, the Dyadic CART is the cheapest to compute followed by ORT in terms of the number of basic operations needed for computation. 
	The computation of the fully penalized LSE, of course, is out of the question. Thus, our risk bounds hint at a natural trade-off 
	between statistical risk computational complexity. 
	
	\smallskip


	\item 
	Gaussian nature of the noise $Z$ is not essential in the above theorem. The conclusion of the theorem holds as long as the entries of $Z$ are independent and are sub gaussian random variables. For the sake of clean exposition, we prove the theorem and its applications in the next section, for the Gaussian case only. In fact, the assumption of independence of Gaussian noise can also be relaxed as discussed in Section~\ref{sec:dep}.
	
	\smallskip
	
	\item The proof (done in Section~\ref{Sec:proofs}) of the above theorem uses relatively standard techniques from high dimensional statistics and is similar to the usual proof of the $\ell_0$ penalized BIC estimator achieving fast rates in high dimensional linear regression 
	(see, e.g, Theorem $3.14$ in~\cite{rigollet2015high}). As is probably well known, sharp 
	oracle inequality with $\frac{1 + \delta}{1 - \delta}$ replaced by $1$ seems unreachable with 
	the current proof technique. Just to reiterate, the point of proving this theorem is to recognize that such a result would imply near optimal results for function classes that are of interest in this current manuscript.  

	\smallskip
	
	\item Operationally, to derive risk bounds for Dyadic CART or ORT for some function class, Theorem~\ref{thm:adapt} behooves us to use approximation theoretic arguments. Precisely, for a given generic $\theta^*$ in the function class, one needs to understand what is the approximation error in the Euclidean sense, if the approximator $\theta$ is constrained to satisfy $k_{\rdp}(\theta) = k$ or $k_{\hier}(\theta) = k$ for any given integer $k.$ One of the technical contributions of this paper lie in addressing this approximation theoretic question for the three classes of functions considered in this paper. 
	
	\smallskip
	
	\item The risk bounds in Theorem~\ref{thm:adapt} as well as the algorithms for computing our estimators (see Section~\ref{sec:algos} in the supplementary file) can be adapted for any subspace, not only the subspace of degree $r$ polynomials. As long as we consider partitions consisting of axis aligned rectangles, we can choose any subspace of functions to be fitted within each rectangle of the partition. Polynomials are one of the most natural and classical subspace of functions which is why we wrote our results for polynomials. 
	
\end{itemize}

The univariate ORT estimator was rigorously studied in~\cite{boysen2009consistencies} where it is called the \textit{Potts minimizer}. This is because the objective function defining the estimator arises in the Potts model in statistical physics; see~\cite{wu1982potts}. Furthermore, this estimator can be computed in $O(n^3)$ time by dynamic programming as shown in~\cite{winkler2002smoothers}. 
This estimator can be thought of as a $\ell_0$ penalized least squares estimator as it penalizes $k^{(0)}(\theta)$ which is same as the $\ell_0$ norm of the difference vector $D \theta = (\theta_{2} - \theta_{1},\dots,\theta_{n} - \theta_{n - 1}).$ It is known that this estimator, properly tuned, indeed nearly attains the minimax rate risk (for instance, this will be implied by Theorem~\ref{thm:adapt}).

The other approach is to consider the corresponding $\ell_1$ penalized estimator,
\begin{equation*}
	\hat{\theta}_{\ell_1,\lambda} = \argmin_{\theta \in \R^{L_{n,1}}}\big( \|y - \theta\|^2 + \lambda \|D \theta\|_1\big).
\end{equation*}
The above estimator is known as univariate Total Variation Denoising (TVD) estimator and 
sometimes also as Fused Lasso in the literature. This estimator is efficiently computable as this is a convex optimization problem. Recent results 
in~\cite{guntuboyina2020adaptive},~\cite{dalalyan2017tvd} and \cite{ortelli2018} have shown that, when properly tuned, the above estimator is also capable of attaining the minimax rate risk under some minimum length assumptions on $\theta^*$ (see Section~\ref{Sec:compare} for details).

To generalize the second approach to the multivariate setting, it is perhaps natural to consider the multivariate Total Variation Denoising (TVD) estimator (see~\cite{hutter2016optimal},~\cite{sadhanala2016total}). However, in a previous manuscript 
of the authors, it has been shown that there exists a $\theta^* \in L_{2,n}$ with $k^{(0)}(\theta^*) = 2$ such that the risk of the ideally tuned {\em constrained} TVD estimator is lower 
bounded by $C N^{-3/4}$, see Theorem $2.3$ in~\cite{chatterjee2019new}. Thus, even for $d = 
2,$ the TVD estimator \textit{cannot} attain the $\tilde{O}(\frac{k(\theta)}{N})$ rate of 
convergence in general. This fact makes us forego the $\ell_1$ penalized approach and return 
to $\ell_0$ penalized least squares. 


Coming to the general $r \geq 1$ case, the literature on fitting piecewise polynomial 
functions is diverse. Methods based on local polynomials and spline functions abound in the 
literature. However, in general dimensions, we are not aware of any rigorous results of the 
precise nature we desire. In the univariate case, there is a family of computationally efficient estimators which fits piecewise polynomials and attain our goal of nearly achieving the oracle risk as stated 
before. This family of estimators is known as {\em trend filtering}--- first proposed in~\cite{kim2009ell_1} and its statistical properties analyzed in~\cite{tibshirani2014adaptive}. Trend filtering is a higher order generalization of the univariate TVD estimator which penalizes the $\ell_1$ norm of higher order derivatives. A continuous version of these estimators, where discrete derivatives are replaced by continuous derivatives, was proposed much earlier in the statistics literature 
by~\cite{mammen1997locally} under the name {\em locally adaptive regression splines}. The 
desired risk adaptation property (of any order $r$) was established 
in~\cite{guntuboyina2020adaptive}, where it was shown that trend filtering (of order $r$) attains the minimax rate risk whenever the underlying truth 
$\theta^*$ is a \textit{discrete spline} and it satisfies some \textit{minimum length} assumptions. See 
Section~\ref{Sec:compare} for a more detailed discussion. However, to the best of our knowledge, such 
bounds are not available in dimension $2$ and above. To summarize this section we can say that beyond the univariate case, our question does not appear to have been answered.  Our goal here is to start filling this gap in the literature.


\begin{remark}
	\label{remark:scope}
	Although the fixed lattice design setting is commonly studied, recall that in this setting the sample size $N \geq 2 ^d$ whenever $n \geq 2.$ In other words, $N$ is necessarily growing exponentially with $d.$ Thus, the results in this paper are really meaningful in the asymptotic regime where $d$ is some fixed ambient dimension and $n \rightarrow \infty.$ Practically speaking, our algorithms and the statistical risk bounds we will present in this manuscript are meaningful for low to moderate dimensions $d.$ Even when $d = 2$, the question we posed about attaining the minimax rate risk adaptively for all truths in a computationally feasible way seems a nontrivial problem.  
\end{remark}

\subsection{Our Results for ORT}\label{secmpp2}
Recall that $K_{r,d}$ is the dimension of the subspace of $d$ dimensional polynomials with degree at most $r \geq 0.$ An immediate corollary of Theorem $2.1$ is the following.
\begin{corollary}\label{cor:pc}
	There exists an absolute constant $C > 0$ such that by setting $\lambda = C K_{r,d} \:\sigma^2\:\log N$ we have the following risk bound, 
	\begin{equation*}
		MSE(\hat{\theta}^{(r)}_{\hier,\lambda},\theta^*) \leq \frac{C\:K_{r,d}\:\:\sigma^2\:\log N}{N}\:k^{(r)}_{\hier}(\theta^*) + \frac{C\:\sigma^2}{N}.
	\end{equation*}
\end{corollary}


Let us discuss some implications of the above corollary. For ORT of order $r \geq 0$, a risk bound scaling like $O(\frac{k^{(r)}_{\hier}(\theta^*)}{N} \log N)$ is guaranteed 
for all $\theta^*$. Thus, for instance, if the true $\theta^*$ is piecewise constant/linear on some arbitrary unknown hierarchical 
partition of $L_{d, n}$, the corresponding ORT estimator of order $0, 1$ respectively achieves the (near) minimax risk $
O(\frac{k^{(r)}_{\all}(\theta^*)}{N} \log N)$. Although this result is an immediate implication of Theorem~\ref{thm:adapt}, this is the first 
such risk guarantee established for a computationally efficient decision tree estimator in 
general dimensions as far as we are aware of. 

At this point, let us recall that our target is to achieve the ideal upper bound $\tilde O(\frac{k^{(r)}_{\all}(\theta^*)}{N})$ to the MSE for all $\theta^*$ which is attained by the fully penalized LSE. However, it is perhaps not efficiently computable. The best upper bound to the MSE we can get for a computationally efficient estimator is 
$\tilde O(\frac{k^{(r)}_{\hier}(\theta^*)}{N})$ which is attained by the ORT estimator.

A natural question that arises at this point is how much worse is the upper bound for ORT than the upper bound for the fully penalized LS estimator given in Theorem~\ref{thm:adapt}. Equivalently, we know that $k_{\all}^{(r)}(\theta^*) \leq 
k_{\hier}^{(r)}(\theta^*)$ 
in general, but how large can the gap be? There definitely exist partitions which are 
not hierarchical, 
i.e. that is 
$\mathcal{P}_{\hier, d, n}$ is a strict subset of $\mathcal{P}_{\all,d, n}$ as shown in Figure~1.

In the next section we explore general and possibly nonhierarchical partitions of $L_{d,n}$ and state several results which basically imply that ORT incurs MSE at most a constant factor more than the ideal fully penalized LSE for several natural instances of rectangular partitions.

\subsubsection{Arbitrary partitions}\label{sec:non}
The risk bound for ORT in Theorem~\ref{thm:adapt} is in terms of $k_{\hier}(\theta^*).$ We 
would like to convert it into a risk bound involving $k_{\all}(\theta^*).$ A natural way of doing this would be to refine an arbitrary partition into a hierarchical partition and 
then count the number of extra rectangular pieces that arises as a result of this 
refinement. This begs the following question of a combinatorial flavour.

\textit{Can an arbitrary partition of $L_{d,n}$ be refined into a hierarchical partition without increasing the number of rectangles too much?}. 

Fortunately, the above question has been studied a fair bit in the computational/combinatorial geometry literature 
under the name of \textit{binary space partitions}. A binary space 
partition (BSP) is a recursive partitioning scheme for a set of 
objects in space. The goal is to partition the space recursively 
until each smaller space contains only one/few of the original 
objects. The main questions of interest are, given the set of 
objects, the minimal cardinality of the optimal partition 
and an efficient algorithm to compute it. A nice survey of this area, explaining the 
central questions and an overview of known results can be found in~\cite{toth2005binary}. 
We will now leverage some existing results in this area which would yield corresponding risk bounds with the help of Theorem~\ref{thm:adapt}. 

For $d = 2,$ it turns out that any rectangular partition can be refined into a hierarchical one where the number of rectangular pieces at most doubles. The following proposition is due to~\cite{berman2002exact} and states this fact. 

\begin{proposition}[\cite{berman2002exact}]\label{prop:partapp}
	Given any partition $\Pi \in \mathcal{P}_{\all,2, n}$ there exists a refinement $\tilde{\Pi} \in \mathcal{P}_{\hier,2, n}$ such that $|\tilde{\Pi}| \leq 2 |\Pi|.$ As a consequence, for any matrix $\theta \in \R^{n \times n}$ and any nonnegative integer $r$, we have $$k^{(r)}_{\hier}(\theta) \leq 2 k^{(r)}_{\all}(\theta).$$
\end{proposition}

The above proposition applied to Theorem~\ref{thm:adapt} immediately yields the following theorem:
\begin{theorem}
	Let $d = 2.$ There exists an absolute constant $C$ such that by setting $\lambda = C\:K_{r,d}\:\:\sigma^2\:\log N$ we have the following risk bound for $\hat{\theta}_{\hier,\lambda}$:
	\begin{equation*}
		\MSE(\hat{\theta}^{(r)}_{\hier,\lambda},\theta^*) \leq \frac{C\:K_{r,d}\:\:\sigma^2\:\log N}{N}\:k^{(r)}_{\all}(\theta^*) + \frac{C\:\sigma^2}{N}.
	\end{equation*}
\end{theorem}

\begin{remark}
	Thus, in the two dimensional setting $d = 2$, {\rm ORT} fulfills the two objectives of computability in polynomial time and attaining the minimax risk
	rate adaptively for all truths $\theta^*.$ Thus, this completely solves the main 
	question we posed in the two dimensional case. To the best of our knowledge, this is the first result of its kind in the literature. 
\end{remark}

For dimensions higher than $2$; 
the best result akin to Proposition~\ref{prop:partapp} that is available is due to~\cite{hershberger2005binary}.

\begin{proposition}[~\cite{hershberger2005binary}]\label{prop:partapp2}
	Let $d > 2.$ Given any partition $\Pi \in \mathcal{P}_{\all, d, n}$ there exists a refinement $\tilde{\Pi} \in \mathcal{P}_{\hier, d, 
		n}$ such that $|\tilde{\Pi}|  \leq |\Pi|^{\frac{d + 1}{3}}.$ As a consequence, for any array $\theta \in \R^{L_{d, n}}$ and any nonnegative integer $r$, we have 
	$$k^{(r)}_{\hier}(\theta) \leq \big(k^{(r)}_{\all}(\theta)\big)^{\frac{d + 1}{3}}.$$
\end{proposition}

\begin{remark}
	A matching lower bound is also given in~\cite{hershberger2005binary} for the case $d = 3.$ Thus, to refine a rectangular partition (of $k$ pieces) into a hierarchical one, one necessarily increases the number of rectangular pieces to $O(k^{4/3})$ in the worst case.   
\end{remark}

The above result suggests that for arbitrary partitions in $d$ dimensions, our current approach will not yield the near minimax rate of convergence. Nevertheless, we state our risk bound that is implied by Proposition~\ref{prop:partapp2}.

\begin{theorem}
	Let $d > 2.$ There exists an absolute constant $C$ such that by setting $\lambda \geq C\:K_{r,d} \:\sigma^2\:\log N$ we have the following risk bound for $\hat{\theta}_{\hier,\lambda}$:
	\begin{equation*}
		\MSE(\hat{\theta}^{(r)}_{\hier,\lambda},\theta^*) \leq \lambda \frac{ \big(k^{(r)}_{\all}(\theta^*)\big)^{\frac{d + 1}{3}}}{N} + \frac{C\:\sigma^2}{N}.
	\end{equation*}
\end{theorem}


Our approach of refining an arbitrary partition into a hierarchical partition does not seem to yield the $\tilde O\big(\sigma^2 \frac{k^{(r)}_{\all}(\theta^*)}{N} \big)$ rate of 
convergence for ORT in dimension higher than $2$ when the truth is a piecewise polynomial 
function on an \textit{arbitrary} rectangular partition. Rectangular partitions in higher 
dimensions could be highly complex; with some rectangles being very ``skinny'' in some 
dimensions. However, it turns out that if we rule out such anomalies, then it is still 
possible to attain our objective. Let us now define a class of partitions which rules out 
such anomalies. 

Let $R$ be a rectangle defined as $R = \Pi_{i = 1}^{d} [a_i,b_i] \subset L_{d,n}.$ Let the sidelengths of $R$ be defined as $n_i = 
b_i - a_i + 1$ for $i \in [d].$ Define its {\em aspect ratio} as $\mathcal{A}(R) = \max \{\frac{n_i}{n_j}: (i,j) \in [d]^2\}.$ For 
any $\alpha \geq 1$, let us call a rectangle $\alpha$ {\em fat} if 
we have $\mathcal{A}(R) \leq \alpha.$ Now consider a rectangular 
partition $\Pi \in \mathcal{P}_{\all,d,n}.$ We call $\Pi$ an \textit{$\alpha$ fat partition} if each of its constituent 
rectangles is $\alpha$ fat. Let us denote the class of $\alpha$ fat partitions of $L_{d,n}$ as $\mathcal{P}_{\fat(\alpha),d,n}.$ As 
before, we can now define the class of subspaces $S^{(r)}_{\fat(\alpha),d,n}$ corresponding to the set of partitions 
$\mathcal{P}_{\fat(\alpha),d,n}.$ 
For any array $\theta^*$ and any integer $r  >0$ we can also denote
\begin{equation*}
	k^{(r)}_{\fat(\alpha)}(\theta^*) = k_{S^{(r)}_{\fat(\alpha),d,n}}(\theta^*).
\end{equation*}


An important result in the area of binary space partitions is that any fat rectangular partition of $L_{n,d}$ can be refined into a hierarchical one with the number of rectangular pieces inflated by at most a constant factor. This is the content of the following proposition which is due to~\cite{de1995linear}.
\begin{proposition}[\cite{de1995linear}]\label{prop:fatpartapp}
	There exists a constant $C(d,\alpha) \geq 1$ depending only on $d$ and $\alpha$ such that any partition $\Pi \in \mathcal{P}_{\fat(\alpha),d,n}$ can be refined into a hierarchical partition $\tilde{\Pi} \in \mathcal{P}_{\hier,d,n}$ satisfying
	\begin{equation*}
		|\tilde{\Pi}| \leq C(d,\alpha) |\Pi|.
	\end{equation*}
	Equivalently, for any $\theta \in \R^{L_{n,d}}$ and any non negative integer $r$ we have
	\begin{equation*}
		k^{(r)}_{\hier}(\theta) \leq C(d,\alpha) k^{(r)}_{\fat(\alpha)}(\theta).
	\end{equation*}
\end{proposition}


The above proposition gives rise to a risk bound for ORT in all dimensions.
\begin{theorem}\label{thm:fat}
	For any dimension $d$ there exists an absolute constant $C$ such that by setting $\lambda \geq C\:K_{r,d}\:\sigma^2\:\log n$ we have the following risk bound for $\hat{\theta}_{\hier,\lambda}$:
	\begin{equation*}
		\E \|\hat{\theta}^{(r)}_{\hier,\lambda} - \theta^*\|^2 \leq \inf_{\theta \in \R^{L_{n,d}}} \big(2\:\|\theta - \theta^*\|^2 + \lambda \:C(d,\alpha)\: k^{(r)}_{\fat(\alpha)}(\theta)\big) + C\:\sigma^2.
	\end{equation*}
\end{theorem}


\begin{remark}
	For any fixed dimension $d$, when $\theta^*$ is piecewise polynomial of degree $r$ on a fat paritition, the above theorem implies a $O\big(\sigma^2 \frac{k^{(r)}_{\all}(\theta^*)}{N} \log N\big)$ bound to the MSE of the ORT estimator (of order $r$). Thus, for arbitrary fat partitions in any dimension, ORT attains our objective of enjoying the near minimax rate of convergence and being computationally efficient. For any fixed dimension $d$, this is the first result of its type that we are aware of. 
\end{remark}

\begin{remark}
	It should be mentioned here that the constant $C(d,\alpha)$ scales exponentially with $d$, at least in the construction which is due to~\cite{de1995linear}. In any case, recall that all of our results are meaningful when $d$ is low to moderate.
\end{remark}

\subsection{Our Results for Dyadic CART}
In the previous section, we showed that the ORT estimator attains the desired $\tilde 
O\big(\sigma^2 \frac{k^{(r)}_{\all}(\theta^*)}{N}\big)$ rate for all $\theta^*$ adaptively 
in dimensions $d = 1,2$ and for all $\theta^*$ which are piecewise polynomial on a fat 
partition in all dimensions $d > 2.$ Since the ORT is more computationally expensive than 
Dyadic CART, a natural question is whether there are analogous results for Dyadic CART. In this case, the relevant question is 
\medskip

\textit{Can an arbitrary nonhierarchical partition of $L_{d,n}$ be refined into a recursive dyadic partition without increasing the number of rectangles too much?}. 
\medskip

When $d = 1$ or $d = 2$, we can give an argument to show there exists a recursive dyadic partition refining a given arbitrary rectangular partition with number of rectangles being multiplied by a log factor. This is the content of our next result which is proved in Section~\ref{sec:dyaproof}. 

\begin{proposition}\label{prop:dyadicref}
	Given any positive integer $n$ and given a partition $\Pi \in \mathcal{P}_{\all,1,n}$ with $k$ intervals, there exists a refinement $\tilde{\Pi} \in \mathcal{P}_{\rdp,1,n}$ which is a recursive dyadic partition with at most $C k \log (en/k)$ intervals where $C > 0$ is an universal constant.
	Equivalently, for all $\theta \in \R^{L_{1,n}}$ and all non negative integers $r$, we have 
	\begin{equation}\label{eq:refine1}
	k^{(r)}_{\rdp}(\theta) \leq C k^{(r)}_{\all}(\theta) \log \frac{en}{k^{(r)}_{\all}(\theta)}.
	\end{equation}

	Moreover, given any positive integer $n$ and an arbitrary partition $\Pi \in \mathcal{P}_{\all,2,n}$ of $L_{2,n}$ with $k$ rectangles there exists a refinement $\Pi^{'} \in \mathcal{P}_{\rdp,2,n}$ which is a recursive dyadic partition  with at most $C k (\log n)^2$ rectangles where $C$ is a universal constant.  Equivalently, for all $\theta \in \R^{L_{2,n}}$ and all non negative integers $r$, we have 
	\begin{equation}\label{eq:refine2}
	k^{(r)}_{\rdp}(\theta) \leq C (\:\log n)^2 k^{(r)}_{\all}(\theta).
	\end{equation}
	
\end{proposition}

We have not seen the above result (equation~\eqref{eq:refine2}) stated explicitly in the Statistics literature. It is probable that this result is known in the combinatorics or computational geometry literature. However, since we could locate an exact reference, we provide its proof in Section~\ref{sec:dyaproof}.

\begin{remark}
	The exponent of $\log n$, which is $1$ for $d = 1$ and $2$ for $d = 2$, cannot be improved 
	in general. It is now natural to conjecture that a result like above is true for a general 
	$d$ where the exponent of $\log n$ is $d.$ However, we do not know whether this is true or 
	not. Our current proof for the $d = 2$ case breaks down and cannot be extended to higher dimensions. See Remark~\ref{rem:dyadicint} for more explanations on this.
\end{remark}

The implication of Proposition~\ref{prop:dyadicref} is the following corollary for Dyadic CART.
\begin{corollary}\label{cor:pcdc}
	For $d = 1$ and any integer $n$, there exists a universal constant $C > 0$ such that by setting $\lambda = C K_{r,1} \:\sigma^2\:\log n$ we have the following risk bound, 
	\begin{equation*}
		MSE(\hat{\theta}^{(r)}_{\rdp,\lambda},\theta^*) \leq C K_{r,1} \sigma^2 \frac{ k^{(r)}_{\all}(\theta^*)}{N} \log \frac{n}{k^{(r)}_{\all}(\theta)} \log n  + \frac{C\:\sigma^2}{N}.
	\end{equation*}
	
	For $d = 2$ and any integer $n$, there exists a universal constant $C > 0$ such that by setting $\lambda = C K_{r,2} \:\sigma^2\:\log n$ we have the following risk bound, 
	\begin{equation*}
		MSE(\hat{\theta}^{(r)}_{\rdp,\lambda},\theta^*) \leq C K_{r,2}\:\sigma^2 \frac{ k^{(r)}_{\all}(\theta^*)}{N} (\log N)^3  + \frac{C\:\sigma^2}{N}.
	\end{equation*}
	
\end{corollary}

To summarize, Dyadic CART attains the same rate as the ORT with an extra $\log N$ factor when $d = 1$ and with an extra $(\log N)^2$ factor when $d = 2.$ We do not know whether for $d > 2$, a result for Dyadic CART analogous to Theorem~\ref{thm:fat} for fat partitions is possible or not.





\section{\textbf{Results for Multivariate Functions with Bounded Total Variation}}\label{sec:tv}

In this section, we will describe an application of Theorem~\ref{thm:adapt} to show that Dyadic CART of order $0$ has near optimal (worst case and adaptive) risk guarantees in any dimension when we consider estimating functions with bounded total variation. Let us first define what we mean by total variation.

Let us think of $L_{d,n}$ as the $d$ dimensional regular lattice graph. Then, thinking of $\theta \in \R^{L_{d,n}}$ as a function on $L_{d,n}$ we define
\begin{equation}\label{eq:TVdef}
\TV(\theta) =  \sum_{(u,v) \in E_{d,n}} |\theta_{u} - \theta_{v}| 
\end{equation}
where $E_{d,n}$ is the edge set of the graph $L_{d,n}.$ 
One way to motivate the above definition is as follows. If we think  $\theta[i_1,\dots,i_n] = f(\frac{i_1}{n},\dots,\frac{i_d}{n})$ for a differentiable function $f: [0,1]^{d} \rightarrow \R$ then the above definition divided by $n^{d - 1}$ is precisely the Reimann approximation for $\int_{[0,1]^d} \|\nabla f\|_1.$ Of course, the definition in~\eqref{eq:TVdef} applies to arbitrary arrays, not just for evaluations of a differentiable function on the grid.  

The usual way to estimate functions/arrays with bounded total variation is to use the Total Variation Denoising (TVD) estimator defined as follows:
$$\hat{\theta}_{\lambda} = \argmin_{\theta \in \R^{L_{d,n}}} \big(\|y - \theta\|^2 + \lambda \TV(\theta)\big).$$ 
This estimator was first introduced in the $d = 2$ case by~\cite{rudin1992nonlinear} for image denoising. This estimator is now a standard and widely succesful technique to do image denoising. In the $d = 1$ setting, it is known (see, e.g.~\cite{donoho1998minimax},~\cite{mammen1997locally}) that a well tuned TVD estimator is minimax rate optimal on the class of all bounded variation signals $\{\theta: \TV(\theta) 
\leq  V\}$ for $ V > 0$. It is also known (e.g, see~\cite{guntuboyina2020adaptive},~\cite{dalalyan2017tvd},~\cite{ortelli2018}) that, when properly tuned, the above estimator is capable of attaining the oracle MSE scaling like $O(\frac{k_{\all}^{(0)}(\theta^*)}{N})$, up to a log factor in $N.$

In the multivariate setting ($d \geq 2$), worst case performance of the TVD estimator has been studied in~\cite{hutter2016optimal},~\cite{sadhanala2016total}. These results show that like in the 1D setting, a well tuned TVD estimator is nearly (up to log factors) minimax rate optimal over the class $\{\theta \in \R^{L_{d,n}}: \TV(\theta) \leq V\}$ of bounded variation signals in any dimension.

The goal of this section is to proclaim that the Dyadic CART estimator $\hat{\theta}^{(0)}_{\rdp,\lambda}$ enjoys similar statistical guarantees as the TVD estimator and possibly even has some advantages over TVD which we list below.

\begin{itemize}
	\item The Dyadic CART estimator $\hat{\theta}^{(0)}_{\rdp,\lambda}$ is computable in $O(N)$ time in low dimensions $d.$ Note that TVD is mostly used for image processing in the $d = 2,3$ case. Recall that the lattice has at least $2^d$ points as soon as $n \geq 2$ so it does not make sense to think of high $d.$ While TVD estimator is the solution of a convex optimization procedure, there is no known algorithm which computes it provably in $O(N)$ time to the best of our knowledge. As we show in Theorem~\ref{thm:tvadapmlb} and Theorem~\ref{thm:dcadap}, the Dyadic CART estimator is also minimax rate optimal over the class $\{\theta \in \R^{L_{d,n}}: \TV(\theta) \leq V\}.$ Thus, the Dyadic CART estimator appears to be the first provably linear time computable estimator achieving the minimax rate, up to log factors, for functions with bounded total variation.

	\item We also show that the Dyadic CART estimator is also adaptive to the intrinsic dimensionality of the true signal $\theta^*.$ We make the meaning of adapting to intrinsic dimensionality precise later in this section. It is not known whether the TVD estimator demonstrates such adaptivity. 

	\item One corollary of Theorem~\ref{thm:adapt} is that the Dyadic CART estimator nearly attains the oracle risk when the truth $\theta^*$ is piecewise constant on a recursive dyadic partition of $L_{n,d}.$ For such signals, the ideally tuned TVD estimator, even in the $d = 2$ case, provably cannot attain the oracle risk for such piecewise constant signals in general; see Theorem $2.3$ in~\cite{chatterjee2019new}.   
\end{itemize}

\subsubsection{Adaptive Minimax Rate Optimality of Dyadic CART}
We now describe risk bounds for the Dyadic Cart estimator for bounded variation arrays. 
Let us define the following class of bounded variation arrays:
\begin{equation*}
	K_{d,n}(V) = \{\theta \in L_{d,n}: \TV(\theta) \leq V\}
\end{equation*}
For any generic subset $S \subset [d]$, let us denote its cardinality by $|S|.$ For any vector $x \in [n]^d$ let us define $x_S \in [n]^{|S|}$ to be the vector $x$ restricted to the coordinates given by $S.$ We now define
\begin{equation*}
	K^{S}_{d,n}(V) = \{\theta \in K_{d,n}(V): \theta(x) = \theta(y) \:\forall x,y\: \in \:[n]^d  \:\:\text{with}\:\: x_S = y_S\}
\end{equation*}
In words, $K^{S}_{n,d}(V)$ is just the set of arrays in $K_{d,n}(V)$ which are a function of the coordinates within $S$ only. In this section, we will show that the Dyadic CART estimator is minimax rate optimal (up to log factors) over the parameter space $K^{S}_{d,n}(V)$ simultaneously over all subsets $S \subset [d].$ This means that the Dyadic CART performs as well as an oracle estimator which knows the subset $S.$ This is what we mean when we say that the Dyadic CART estimator adapts to intrinsic dimensionality. To the best of our knowledge, such an oracle property in variable selection is rare in Non Parametric regression. The work in~\cite{bertin2008selection} shows a two step procedure for adapting to instrinsic dimensionality for multivariate Holder smooth function classes. The only comparable result that we are aware of for a spatially heterogenous function class is Theorem $3$ in~\cite{deng2018isotonic} which proves a similar adaptivity result in multivariate isotonic regression.   

Fix a subset $S \subset [d]$ and let $s = |S|.$ Consider our Gaussian mean estimation problem where it is known that the underlying truth $\theta^* \in K^{S}_{d,n}(V).$ We could think of $\theta^*$ as $n^{d - s}$ copies of a $s$ dimensional array $\theta^*_{S} \in \R^{L_{s,n}}.$ It is easy to check that $\theta^*_{S} \in K_{s,n}(V_S)$ where $V_s = \frac{V}{n^{d - s}}.$ Estimating $\theta^*$ is equivalent to estimating the $s$ dimensional array $\theta^*_{S}$ where the noise variance is now reduced to $\sigma^2_{S} = \frac{\sigma^2
}{n^{d - s}}$ because we can average over $n^{d - s}$ elements per each entry of $\theta^*_{S}.$ Therefore, we now have a reduced Gaussian mean estimation problem where the noise variance is $\sigma^2_{S}$ and the parameter space is $K_{n,s}(V_S).$ A tight lower bound to the minimax risk for the parameter space $K_{d,n}(V)$ for arbitrary $n,d,V > 0$ is available in~\cite{sadhanala2016total}. Using the above logic and this existing minimax lower bound allows us to establish a lower bound to the minimax risk for the parameter space $K^{S}_{d,n}(V).$ The detailed proof is given in Section~\ref{Sec:proofs}.

\begin{theorem}[Minimax Lower Bound over $K^{S}_{d,n}(V)$]
	\label{thm:tvadapmlb}
	Fix positive integers $n,d$ and let $S \subset [d]$ such that $s = |S| \geq 2.$ Let $V > 0$ and $V_S = \frac{V}{n^{d - s}}.$ Similarly, for $\sigma > 0,$ let $\sigma^2_S = \frac{\sigma^2}{n^{d - s}}.$ There exists a universal constant $c > 0$ such that 
	\begin{equation*}
		\inf_{\tilde{\theta} \in \R^{L_{d, n}}} \sup_{\theta \in K^{S}_{d,n}(V)} \E_{\theta} \|\tilde{\theta} - \theta\|^2 \geq c\:n^{d - s}\: \min\{\frac{\sigma_{S}\:V_S}{2s} \sqrt{1 + \log(\frac{2\:\sigma\:s\:n^s}{V_S})}, n^{s} \sigma_S^2, \frac{V_S}{s}^2 + \sigma_S^2\}.
	\end{equation*}
	If $|S| = 1$ then 
	\begin{equation*}
		\inf_{\tilde{\theta} \in \R^{L_{d, n}}} \sup_{\theta \in K^{S}_{d,n}(V)} \E_{\theta} \|\tilde{\theta} - \theta\|^2 \geq c\:n^{d - 1}\: \min\{(\sigma^2_{S} V_{S})^{2/3} n^{1/3}, n\:\sigma^2_{S}, n\:V^2_{S}\}.
	\end{equation*}
\end{theorem}

Let us now explain the above result. If we take the subset $S = [d]$ this is exactly the lower bound in Theorem $2$ of~\cite{sadhanala2016total}. All we have done is stated the same result for any subset $S$ since we can reduce the estimation problem in $K^{S}_{d,n}(V)$ to a $s$ dimensional estimation problem over $K_{s,n}(V_S).$ The bound is in terms of a minimum of three terms. It is enough to explain this bound in the case when $S = [d]$ as similar reasoning holds for any subset $S$ with $s = |S| \geq 2.$ Thinking of $\sigma$ as a fixed constant, the three terms in the minimum on the right side corresponds to different regimes of $V.$ It can be shown that the constant array with each entry $\overline{y}$ attains the $V^2 + \sigma^2$ rate which is dominant when $V$ is very small. The estimator $y$ itself attains the $N \sigma^2$ rate which is dominant when $V$ is very large. Hence, these regimes of $V$ can be thought of as trivial regimes. In the nontrivial regime, the lower bound is $c\: \min\{\frac{\sigma\:V}{2d} \sqrt{1 + \log(\frac{2\:\sigma\:d\:N}{V})}\}.$ 


It is also known that a well tuned TVD estimator is minimax rate optimal, in the nontrivial regime, over $K_{d,n}(V)$ for all $d \geq 2$, up to log factors; see~\cite{hutter2016optimal}. For instance, it achieves the above minimax lower bound (up to log factors) in the nontrivial regime. For this reason, we can define an oracle estimator (which knows the set $S$) attaining the minimax lower bound over $K_{d,n}^{S}(V)$ in Theorem~\ref{thm:tvadapmlb}, up to log factors. The oracle estimator would first obtain $\overline{y}_{S}$ by averaging the observation array $y$ over the coordinates in $S^{C}$ and then it would apply the $s$ dimensional TVD estimator on $\overline{y}_{S}.$ Our main point here is that the Dyadic CART estimator performs as well as this oracle estimator, without the knowledge of $S$. In other words, its risk nearly (up to log factors) matches the minimax lower bound in Theorem~\ref{thm:tvadapmlb} adaptively over all subsets $S \subset [d].$ 
This is the content of our next theorem which is proved in Section~\ref{Sec:proofs} (in the supplementary file). 

\begin{theorem}[Adaptive Risk Bound for Dyadic Cart]
	\label{thm:dcadap}
	Fix any positive integers $n,d.$ Let $\theta^* \in K^{S}_{d,n}(\infty)$ be the underlying truth where $S \subset [d]$ is any subset with $|S| \geq 2.$ Let $V^* = \TV(\theta^*).$
	Let $V^*_{S} = \frac{V^*}{n^{d - s}}$ and $\sigma^2_{S} = \frac{\sigma^2}{n^{d - s}}$ be defined as before. The following risk bound holds for the Dyadic CART estimator $\hat{\theta}^{(0)}_{\rdp,\lambda} $ with $\lambda \geq C \sigma^2 \log N$ where $C$ is an absolute constant.
	\begin{equation*}
		\E_{\theta^*} \|\hat{\theta}^{(0)}_{\rdp,\lambda} - \theta^*\|^2 \leq C\:n^{d - s}\:\min\{ \sigma_{S} V^*_{S} \log N, \sigma_{S}^2 \log N, \big((V^*_{S})^{2} + \sigma_{S}^2\big)\}
	\end{equation*}
	In the case $|S| = 1$ we have
	\begin{equation*}
		\E_{\theta^*} \|\hat{\theta}^{(0)}_{\rdp,\lambda} - \theta^*\|^2 \leq C\:n^{d - 1}\: \min\{(\sigma^2_{S} V_{S} \log N)^{2/3} n^{1/3}, n\:\sigma^2_{S} \log N, n\:V^2_{S} + \sigma_{S}^2 \log N\}
	\end{equation*}
\end{theorem}



We think the following is an instructive way to read off the implications of the above theorem. Let us consider $d \geq 2$ and the $S = [d]$ case. We will only look at the nontrivial regime even though Dyadic CART remains minimax rate optimal, up to log factors, even in the trivial regimes. In this case, $MSE(\hat{\theta}^{(0)}_{\rdp,\lambda},\theta^*) = \tilde{O}(\frac{\sigma V^*}{N})$ which is the minimax rate in the nontrivial regime as given by Theorem~\ref{thm:tvadapmlb}. Now, for many natural instances of $\theta^*$, the quantity $V^* = O(n^{d - 1})$; for instance if $\theta^*$ are evaluations of a differentiable function on the grid. This $O(n^{d - 1})$ scaling was termed as the \textit{canonical scaling} for this problem by~\cite{sadhanala2016total}. Therefore, under this canonical scaling for $V^*$ we have $$MSE(\hat{\theta}^{(0)}_{\rdp,\lambda},\theta^*) = \tilde{O}(\frac{\sigma}{n}) = \tilde{O}(\frac{\sigma}{N^{1/d}}).$$

Now let us consider $d \geq 2$ and a general subset $S \subset [d].$ In the nontrivial regime, by Theorem~\ref{thm:dcadap} we have $MSE(\hat{\theta}^{(0)}_{\rdp,\lambda},\theta^*) = \tilde{O}(\frac{\sigma_S V^*_{S}}{n^{s}})$ which is also the minimax rate over the parameter space $K_{d,n}^{S}.$ Now, $V^*_{S} = O(n^{s - 1})$ under the canonical scaling in this case. Thus, under this canonical scaling we can write $$MSE(\hat{\theta}^{(0)}_{\rdp,\lambda},\theta^*) = \tilde{O}(\frac{\sigma_S}{n}) = \tilde{O}(\frac{\sigma_S}{N^{1/d}}).$$ 
This is very similar to the last display except $\sigma$ has been replaced by $\sigma_{S}$, the \textit{actual standard deviation} of this problem. The point is, the Dyadic CART attains this rate without knowing $S$. The case when $|S| = 1$ can be read off in a similar way.


\section{\textbf{Results for Univariate Functions of Bounded Variation of Higher Orders}}\label{Sec:uni}
In this section, we show another application of Theorem~\ref{thm:adapt} to a family of univariate function classes which have been of recent interest. The results in this section would be for the univariate Dyadic Cart estimator of some order $r \geq 0.$ As mentioned in Section~\ref{Sec:intro}, TV denoising in the 1D setting has been studied as part of a general family of estimators which penalize discrete derivatives of different orders. These estimators have been studied in~\cite{mammen1997locally} ,~\cite{steidl2006splines},~\cite{tibshirani2014adaptive},~\cite{guntuboyina2020adaptive} and~\cite{kim2009ell_1} who coined the name trend filtering. 


To define the trend filtering estimators here, we first need to define variation of all orders. For a vector $\theta \in \R^n,$ let us define $D^{(0)}(\theta) = \theta, D^{(1)}(\theta) = (\theta_2 - \theta_1,\dots,\theta_n - \theta_{n - 1})$ and $D^{(r)}(\theta)$, for $r \geq 2$, is recursively defined as $D^{(r)}(\theta) = D^{(1)}(D^{(r - 1)}(\theta)).$ Note that $D^{(r)}(\theta) \in \R^{n - r}.$ For simplicity, we denote the operator $D^{(1)}$ by $D.$ For any positive integer $r \geq 1$, let us also define the $r$ th order variation of a vector $\theta$ as follows:
\begin{equation}
V^{(r)}(\theta) = n^{r - 1} |D^{(r)}(\theta)|_{1}
\end{equation}
where $|.|_1$ denotes the usual $\ell_1$ norm of a vector. Note that $V^{(1)}(\theta)$ is the usual total variation of a vector as defined in~\eqref{eq:TVdef}.
\begin{remark}
	The $n^{r - 1}$ term in the above definition is a normalizing factor and is written following the convention adopted in~\cite{guntuboyina2020adaptive}. If we think of $\theta$ as evaluations of a $r$ times differentiable function $f:[0,1] \rightarrow \R$ on the grid $(1/n,2/n\dots,n/n)$ then the Reimann approximation to the integral $\int_{[0,1]} f^{(r)}(t) dt$ is precisely equal to $V^{(r)}(\theta).$ Here $f^{(r)}$ denotes the $r$th derivative of $f.$ Thus, for natural instances of $\theta$, the reader can imagine that $V^{(r)} = O(1).$ 
\end{remark}

Let us now define the following class of sequences for any  integer $r \geq 1$,
\begin{equation}
\mathcal{BV}^{(r)}_{n}(V) = \{\theta \in \R^n: V^{(r)}(\theta) \leq V\}.
\end{equation}

Trend Filtering (of order $r \geq 1$) estimators are defined as follows for a tuning parameter $\lambda > 0$:
\begin{equation*}
	\hat{\theta}^{(r)}_{tf,\lambda} = \argmin_{\theta \in \R^n} \big(\|y - \theta\|^2 + \lambda  V^{(r)}(\theta)\big).
\end{equation*}
Thus, Trend Filtering is penalized least squares where the penalty is proportional to the $\ell_1$ norm of $D^{(r)}(\theta).$ As opposed to Trend Filtering, here we will study the univariate Dyadic CART estimator (of order $r - 1$) which penalizes something similar to the $\ell_0$ norm of $D^{(r)}(\theta).$ We will show that Dyadic CART (of order $r - 1$) compares favourably with Trend Filtering (of order $r$) in the following aspects:

\begin{itemize}
	
	\item For a given constant $V > 0$ and $r \geq 1$, $n^{-2r/2r + 1}$ rate is known to be the minimax rate of estimation over the space $\mathcal{BV}^{(r)}_{n}(V)$; (see e.g,~\cite{donoho1994ideal}). A standard terminology in this field terms this $n^{-2r/2r + 1}$ rate as the \textit{slow rate}. It is also known that a well tuned Trend Filtering estimator is minimax rate optimal over the parameter space $\mathcal{BV}^{(r)}_{n}(V)$ and thus attains the slow rate. This result has been shown in~\cite{tibshirani2014adaptive} and~\cite{wang2014falling} building on earlier results by~\cite{mammen1997locally}. In Theorem~\ref{thm:slowrate} we show that Dyadic CART estimator of order $r - 1$ is also near minimax rate optimal (up to a log factor) over the space $\mathcal{BV}^{(r)}_{n}(V)$ and attains the slow rate.

	
	\item It is also known that an ideally tuned Trend Filtering (of order $r$) estimator can adapt to $\|D^{r}(\theta)\|_0$, the number of jumps in the $r$ th order differences, \textit{under some assumptions on $\theta^*$}. Such a result has been shown in~\cite{guntuboyina2020adaptive} and~\cite{van2019prediction}. In this case, the Trend Filtering estimator of order $r$ attains the $\tilde{O}(\|D^{(r)}(\theta)\|_0/n)$ rate. Standard terminology in this field terms this as the \textit{fast rate}. In Theorem~\ref{thm:fastrate} we show that Dyadic CART estimator of order $r - 1$ attains the fast rate \textit{without any assumptions} on $\theta^*.$

	\item If one desires the fast rate for both piecewise constant and piecewise linear functions, there is no way to attain this by a single Trend Filtering Estimator. One needs to use Trend Filtering of order $r = 2$ to attain fast rates for piecewise linear functions and $r = 1$ for piecewise constant functions respectively. In contrast, Dyadic Cart of order $1$ attains the fast rate for both piecewise linear and piecewise constant functions. This is because by our definition, a piecewise constant function is also piecewise linear. In general, Dyadic Cart of order $r$ attains fast rate for a piecewise polynomial of degree $r$ where each piece has degree $\leq r.$

	\item To the best of our knowledge, different tuning parameters for Trend Filtering are needed depending on whether slow rate or fast rate is desired. Thus, technically the estimators giving the slow rate and the fast rate are different. In contrast, the same tuning parameter gives both the slow rate and the fast rate for Dyadic CART.

	\item The univariate Dyadic CART estimator of order $r \geq 0$ can be computed in linear $O(N)$ time. Although Trend Filtering estimators are efficiently computable by convex optimization, we are not aware of a provably $O(N)$ run time bound (for general $r \geq 1$) on its computational complexity.
\end{itemize}

\begin{remark}
	By Theorem~\ref{thm:adapt}, the univariate ORT would also satisfy all the risk bounds that we prove for univariate Dyadic Cart in this section. Recall that for $r = 0$, the ORT is precisely the same as the fully penalized least squares estimator $\hat{\theta}_{\all,\lambda}$ studied in~\cite{boysen2009consistencies}. This is because $\mathcal{P}_{\hier,n,1}$ coincides with $\mathcal{P}_{\all,1,n}$ as all univariate partitions are hierarchical. Since the computational complexity of univariate Dyadic Cart is $O(N)$ and of univariate $ORT$ is $O(N^3)$ we focus on univariate Dyadic Cart.
\end{remark}

\subsubsection{Risk Bounds for Univariate Dyadic CART of all orders}

We start with the bound of $n^{-2r/(2r + 1)}$ for the risk of Dyadic CART of order $r - 1$ for the parameter space $\mathcal{BV}^{(r)}_{n}(V).$ We also explicitly state the dependence of the bound on $V$ and $\sigma.$ 

\begin{theorem}[Slow Rate for Dyadic CART]\label{thm:slowrate}
	Fix a positive integer $r.$ Let $V^{r}(\theta^*) = V.$ For the same constant $C$ as in Theorem~\ref{thm:adapt}, if we set $\lambda \geq C \sigma^2 \log n$ we have
	\begin{equation}
	MSE(\hat{\theta}^{(r - 1)}_{\rdp,\lambda},\theta^*) \leq C_r \big(\frac{\sigma^2 V^{1/r} \log n}{n}\big)^{2r/(2r + 1)} + C_r \sigma^2 \frac{\log n}{n}	
	\end{equation}
	where $C_r$ is an absolute constant only depending on $r.$
\end{theorem}

\begin{remark}
	The proof of the above theorem is done in Section~\ref{Sec:proofs} (in the supplementary file). The proof proceeds by approximating any $\theta \in\mathcal{BV}^{(r)}_{n}(V)$ with a vector $\theta^{'}$ which is piecewise polynomial of degree $r - 1$ with an appropriate bound on its number of pieces and then invoking Theorem~\ref{thm:adapt}.
\end{remark}

\begin{remark}
	The above theorem shows that the univariate Dyadic CART estimator of order $r - 1$ is minimax rate optimal up to the $(\log n)^{2r/(2r + 1)}$ factor. The dependence of $V$ is also optimal in the above bound. Up to the log factor, this upper bound matches the bound already known for the Trend Filtering estimator of order $r$; (see e.g, ~\cite{tibshirani2014adaptive}). 
\end{remark}

Our next bound shows that the univariate Dyadic CART estimator achieves our goal of attaining the oracle risk for piecewise polynomial signals. 

\begin{theorem}[Fast Rates for Dyadic CART]\label{thm:fastrate}
	Fix a positive integer $r$ and $0 < \delta < 1.$ Let $V^{r}(\theta^*) = V.$ For the same constant $C$ as in Theorem~\ref{thm:adapt}, if we set $\lambda \geq C \sigma^2 \log n$ we have
	\begin{equation*}
		\E \|\hat{\theta}^{(r)}_{\rdp,\lambda} - \theta^*\|^2 \leq \inf_{\theta \in \R^N} \big[\frac{(1 + \delta)}{(1 - \delta)}\:\|\theta - \theta^*\|^2 + \frac{\lambda C_r}{1 - \delta}\:k^{(r)}_{\all}(\theta)\:\log (\frac{en}{k^{(r)}_{\all}(\theta)})\big] + C \frac{\sigma^2}{\delta\:(1 - \delta)}	
	\end{equation*}
	where $C_r$ is an absolute constant only depending on $r.$ As a corollary we can conclude that 
	\begin{equation*}
		MSE(\hat{\theta}^{(r)}_{\rdp,\lambda},\theta^*) \leq C_r \sigma^2 \frac{k^{(r)}_{\all}(\theta^*) \log n \log (\frac{en}{k^{(r)}_{\all}(\theta^*)})}{n}.
	\end{equation*} 
\end{theorem}

\begin{proof}
	The proof follows directly from the risk bound for univariate Dyadic Cart  given in Theorem~\ref{thm:adapt} and applying equation~\eqref{eq:refine1} in Lemma~\ref{lem:1dpartbd} which says that $k^{(r)}_{\rdp}(\theta) \leq k^{(r)}_{\all}(\theta) \log (\frac{en}{k^{(r)}_{\all}(\theta)})$ for all vectors $\theta \in \R^n.$ 
\end{proof}

Let us now put our result in Theorem~\ref{thm:fastrate} in context. It says that in the $d = 1$ case, Dyadic CART achieves our goal of attaining MSE scaling like $\tilde{O}(k^{(r)}_{\all}(\theta^*)/n)$ (fast rate) for all $\theta^*.$ The Trend Filtering estimator, ideally tuned, is also capable of attaining this rate of convergence; (see Theorem $3.1$ in~\cite{van2019prediction} and Theorem $2.1$ in~\cite{guntuboyina2020adaptive}), under certain \textit{minimum length conditions} on $\theta^*.$  Let us discuss this issue now in more detail and compare Theorem~\ref{thm:fastrate} to the comparable result known for Trend Filtering. 
\subsubsection{Comparison of Fast Rates for Trend Filtering and Dyadic CART}\label{Sec:compare}
The fast rate results available for Trend Filtering give bounds on the MSE of the form $\tilde{O}(|D^{(r)}(\theta^*)|_0/n)$ where $|.|_0$ refers to the number of nonzero elements of a vector. Now for every vector $\theta \in R^n$, $|D^{(r)}(\theta)|_0 = k$ 
if and only if $\theta$ equals $(f(1/n), . . . , f(n/n))$ for a discrete spline function f that is made of $k + 1$ polynomials each of degree at most $r - 1$. Discrete splines are piecewise polynomials with regularity at the knots. They differ from the usual (continuous) splines in the form of the regularity condition at the knots: for splines, the regularity condition translates to
(higher order) derivatives of adjacent polynomials agreeing at the knots, while for discrete splines it translates to discrete differences of adjacent polynomials agreeing at the knots; see~\cite{mangasarian1971discrete} for details. This fact about the connection between $|D^{(r)}(\theta^*)|_0$ and discrete splines is standard (see e.g.,~\cite{steidl2006splines}) and a proof can be found in Proposition D.3 in~\cite{guntuboyina2020adaptive}.

The above discussion then directly implies for any $\theta \in \R^n$ and any $r \geq 1$,
\begin{equation}
k^{(r - 1)}_{\all}(\theta) \leq |D^{(r)}(\theta)|_0 + 1.
\end{equation}

Therefore any bound of the form $\tilde{O}(k^{(r - 1)}_{\all}(\theta^*)/n)$ is automatically also $\tilde{O}(|D^{(r)}(\theta^*)|_0/n).$ Thus, Theorem~\ref{thm:fastrate} implies that the Dyadic CART attains the fast rate whenever Trend Filtering does so. However, as we now argue, there is a class of functions for which the Dyadic CART attains the fast rate but Trend Filtering does not.

A key point is that a minimum length condition needs to hold for Trend Filtering to attain fast rates as explained in~\cite{guntuboyina2020adaptive}. For example, when $r = 1$, consider the sequence of vectors in $\R^n$ of the form $\theta^* = (0,\dots,0,1).$ Clearly, $\theta^*$ is piecewise constant with $2$ pieces. However, to the best of our knowledge, the Trend Filtering estimator (with the tuning choices proposed in the literature such as the ideal tuning for constrained version of Trend Filtering as in~\cite{guntuboyina2020adaptive}) will not attain a $\tilde{O}(1/n)$ rate for this sequence of $\theta^*$ since it needs the length of the constant pieces to be $O(n).$ However, the Dyadic CART estimator does not need any minimum length condition for Theorem~\ref{thm:fastrate} to hold and will attain the $\tilde{O}(1/n)$ rate for this sequence of $\theta^*.$

Now let us come to the case when $r \geq 2.$ It is known that Trend Filtering of order $r$ fits a discrete spline of degree $(r - 1)$. Thus, if the truth is piecewise polynomial with small number of pieces but it does not satisfy regularity conditions such as being a discrete spline, then Trend Filtering cannot estimate well. The reason is that Trend Filtering can only fit discrete splines. On the other hand, as long as the truth is piecewise polynomial with not too many pieces, Dyadic CART does not need the regularity conditions to be satisfied in order to perform well. Let us illustrate this with a simple example in the case when $r = 2.$ Similar phenomena is true for higher $r.$

Let's consider a discontinuous piecewise linear function $f^*:[0,1] \rightarrow \R$ defined as follows:
\begin{equation*}
	f^*(x) =
	\begin{cases}
		x, & \text{for} \:\:x \leq 1/2 \\
		2x & \text{for} \:\:1/2 < x \leq 1
	\end{cases}
\end{equation*}
Let $\theta^* \in \R^n$ such that $\theta^* = (f(1/n),\dots,f(n/n)).$ Clearly, $k^{(1)}_{\all}(\theta^*) = 2$ and by Theorem~\ref{thm:fastrate}, the Dyadic CART estimator of order $1$ attains the $\tilde{O}(1/n)$ rate for this sequence of $\theta^*.$ If we check the vector $D^{(1)}(\theta^*)$ it is of the form $(a,\dots,a,b,c,\dots,c).$ It is piecewise constant with three pieces. However, it does not satisfy the minimum length condition as the middle piece has length only $1$ and not $O(n).$ Thus, the Trend Filtering estimator of order $2$ won't attain the fast rate for such a sequence of $\theta^*.$ In fact, it can be shown that the Trend Filtering estimator won't even be able to attain the slow rate and would be inconsistent for such a sequence of $\theta^*$ simply because Trend Filtering can only fit discrete splines.

Another point worth reiterating is that to the best of our knowledge, the results for Trend Filtering say that the tuning parameter needs to be set differently depending on whether one wants to obtain the slow rates or the fast rates; see~\cite{guntuboyina2020adaptive} and~\cite{van2019prediction}. In the case of Dyadic CART, both Theorem~\ref{thm:slowrate} and Theorem~\ref{thm:fastrate} hold under the choice of the same tuning parameter. Moreover, Theorem~\ref{thm:fastrate} says that fast rates for Dyadic CART of order $r - 1$ hold whenever the true signal is piecewise polynomial with few pieces and each polynomial can have degree ranging from $0$ to $r - 1.$ This is because for any vector $\theta \in \R^n$, the complexity measure $k^{(r)}(\theta)$ is non increasing in $r.$ This means, if we use Dyadic Cart of order $2$, fast rates are guaranteed for piecewise quadratic, piecewise linear and piecewise constant signals. However, the same is not true for Trend Filtering where the order $r$ needs to be set to be $3,2,1$ depending on whether we want fast rates for piecewise quadratic or linear or constant signals respectively. Among several advantages there seems to be only one disadvantage for Dyadic CART versus Trend Filtering. It is the presence of extra log factors in the risk bounds. All in all, our results indicate that Univariate Dyadic CART may enjoy certain advantages over Trend Filtering in both the statistical and computational aspects.

\begin{remark}
	It should be remarked here that wavelet shrinkage methods with appropriate tuning methods can also attain the slow and fast rates as shown in~\cite{donoho1998minimax},~\cite{donoho1994ideal}. Wavelet shrinkage method can also be computed in $O(n)$ time. However, as is well known, wavelet methods require $n$ to be a power of $2$ and often there are boundary effects that need to be addressed for the fitted function. Univariate Dyadic CART seems related to wavelet shrinkage as both arise from dyadic thinking but they are different estimators. The way Dyadic CART has been defined in this article, $n$ does not need to be a power of $2$ and no boundary effects appear for Dyadic CART. In any case, our point here is not to compare Dyadic CART with wavelet shrinkage but to demonstrate the efficacy of Dyadic CART in fitting piecewise polynomials. 
\end{remark}


\section{Simulations}\label{sec:sim}
In this section, we present numerical evidence for our theoretical results. In our simulations we generated data from ground truths $\theta^*$ with certain size and did monte carlo repetitions to estimate the MSE. We also fitted a least squares line to log MSE versus $\log n$. This slope is supposed to give us some indication about the exponent of $N$ in the rate of convergence of the MSE to 0. To set the tuning parameter $\lambda$, we did not do very systematic optimization. Rather, we made a choice which gave us reasonable results. To implement ORT and Dyadic CART, we wrote our own code in R. Our codes are very basic and it is likely that the run times can be speeded up by more efficient implementations. All our simulations are completely reproducible and our codes are available on request.

\subsection{Simulation for two dimensional ORT}
We take a ground truth $\theta^*$ of size $n \times n$ which is piecewise constant on a nonhierarchical partition as shown in Figure~\ref{fig2}.
\begin{figure}
	\begin{center}
		\includegraphics[scale=0.75]{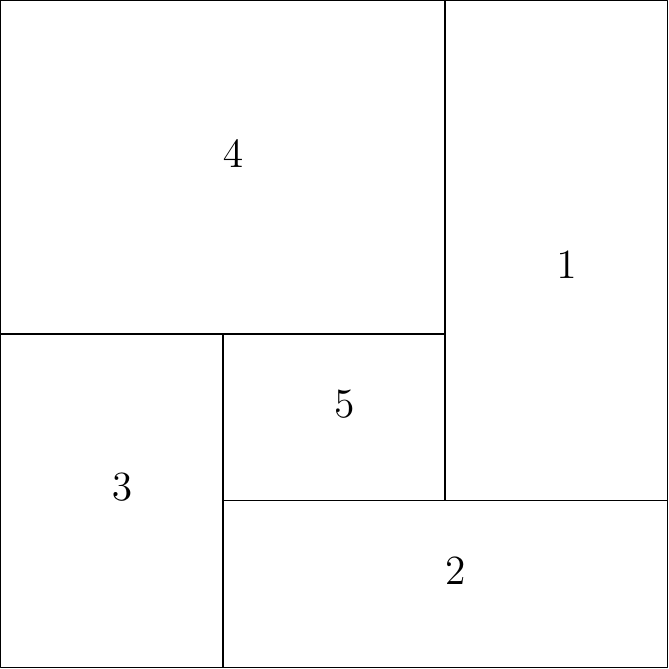} 
	\end{center}
	\caption{\textit{Figure depicts the ground truth matrix which is piecewise constant on a non hierarchical partition.}} 
	\label{fig2}
\end{figure}
We varied $n$ in increments of $5$ going from $30$ to $50$. We generated data $y$ by adding mean $0$ Gaussian noise with standard deviation $0.1$ and then applied ORT (of order $0$) to $y.$ We replicated this experiment $50$ times for each $n.$ We set the tuning parameter $\lambda$ to be increasing from $0.1$ to $0.18$ in increments of $0.02.$ 

For the sake of comparison, we also implemented the constrained version of two dimensional total variation denoising with ideal 
tuning (i.e. setting the constraint to be equal to the actual total variation (TVD) of the truth $\theta^*$). In the low $\sigma$ limit, it can be proved that this ideally tuned constrained 
estimator is better than the corresponding penalized estimator for 
every deterministic choice of the tuning parameter. 
This follows from the results of~\cite{oymak2013sharp} as described in Section 5.2 in~\cite{guntuboyina2020adaptive}. In this sense, we are comparing with the best possible version of TVD.

\begin{figure}
	\begin{center}
		\includegraphics[scale=0.25]{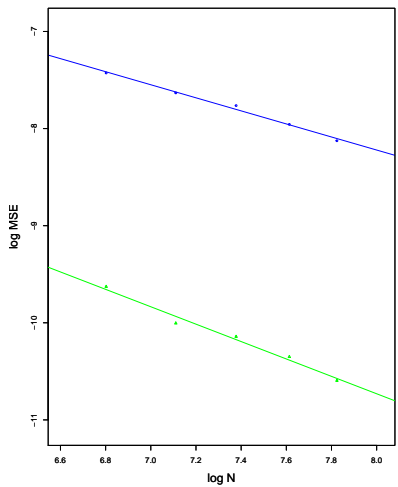} 
	\end{center}
	\caption{\textit{This is a $\log$ MSE vs $\log N$ plot for ORT (in green) and ideal TVD (in blue). The slope for ideal TVD comes out to be $-0.67$ and and the slope for ORT comes out to be $-0.9$}} 
	\label{fig3}
\end{figure}

From Figure~\ref{fig3} we see that ORT outperforms the ideal TVD. The slope of the least squares line is close to $-1$ for ORT which agrees with our $\log N/N$ rate predicted by our theory for ORT. The slope of the least squares line also agrees with the $\tilde{O}(N^{-3/4})$ rate predicted by Theorem $2.3$ in~\cite{chatterjee2019new}. It should be mentioned here that we did not take $n$ larger than $50$ because our implementation in R of ORT is slow (one run with $n = 50$ takes $8$ minutes). Recall that the computational complexity of ORT is $O(n^5)$ in this setting. We believe that more efficient implementations might make ORT practically computable for $n$ in the hundreds. We also chose the standard deviation $\sigma = 0.1$ because for larger $\sigma$ one needs larger sample size to see the rate of convergence of the MSE.


\subsection{Simulation for two dimensional Dyadic CART}
Here we compare the performance of Dyadic CART of order $0$ and the constrained version of two dimensional total variation denoising with ideal tuning. 

\subsubsection{Two piece matrix}
We consider the simplest piecewise constant matrix $\theta^* \in \R^{n \times n}$ matrix where $\theta^*(i,j) = \I\{j \leq n/2\}.$ Hence $\theta^*$ just takes two distinct values and the true rectangular partition is dyadic. Thus, this is expected to be a favourable case for Dyadic CART. We generated data by adding a matrix of independent standard normals. We took a sequence of $n$ geometrically increasing from $2^4$ to $2^9$. We chose $\lambda = \ell$ when $n = 2^{\ell}.$ 


Our simulations (see Figure~\ref{fig4}) suggest that in this case, the ideally tuned constrained TVD estimator is outperformed by Dyadic CART in terms of statistical risk. The least squares slope for TVD comes out to be $-0.71$. Theoretically, it is known that the rate of convergence for ideal constrained TVD is actually $N^{-3/4}$; see Theorem $2.3$ in~\cite{chatterjee2019new}. The least squares slope for Dyadic CART comes out to be $-1.26$. Of course, we expect the actual rate of convergence for Dyadic CART to be $\tilde{O}(N^{-1})$ in this case. The constrained TVD estimator was computed by the convex optimization software MOSEK (via the R package
Rmosek).  For $n = 2^9 = 512$, Dyadic CART was much faster to compute and there was a significant difference in the runtimes. Actually, we did not take $n$ larger than $2^9$ because the RMosek implementation of TVD was becoming too slow. However, with our implementation of Dyadic Cart, we could run it for sizes as large as $2^{15} \times 2^{15}.$ 

\subsubsection{Smooth Matrix}
We considered the matrix $\theta^* \in \R^{n \times n}$ matrix where $\theta^*(i,j) = \sin(i\:\pi/n) \sin(j\:\pi/n).$ We generated data by adding a matrix of independent standard normals. We again took a sequence of $n$ geometrically increasing from $2^4$ to $2^9$ and chose $\lambda = \ell$ when $n = 2^{\ell}.$ The ground truth here is a smooth matrix which is expected to favour TVD more than Dyadic CART. In this case, we saw (see Figure~\ref{fig4}) that the slopes of the least squares line came out to be around $-0.55$ for both Dyadic CART and TVD. Recall that $\tilde{O}(N^{-0.5})$ rate is the minimax rate for bounded variation functions. The ideally tuned TVD did have a slightly lower MSE than Dyadic CART for our choice of $\lambda$ in this example.

\begin{figure}
	\begin{center}
		\includegraphics[scale=0.25]{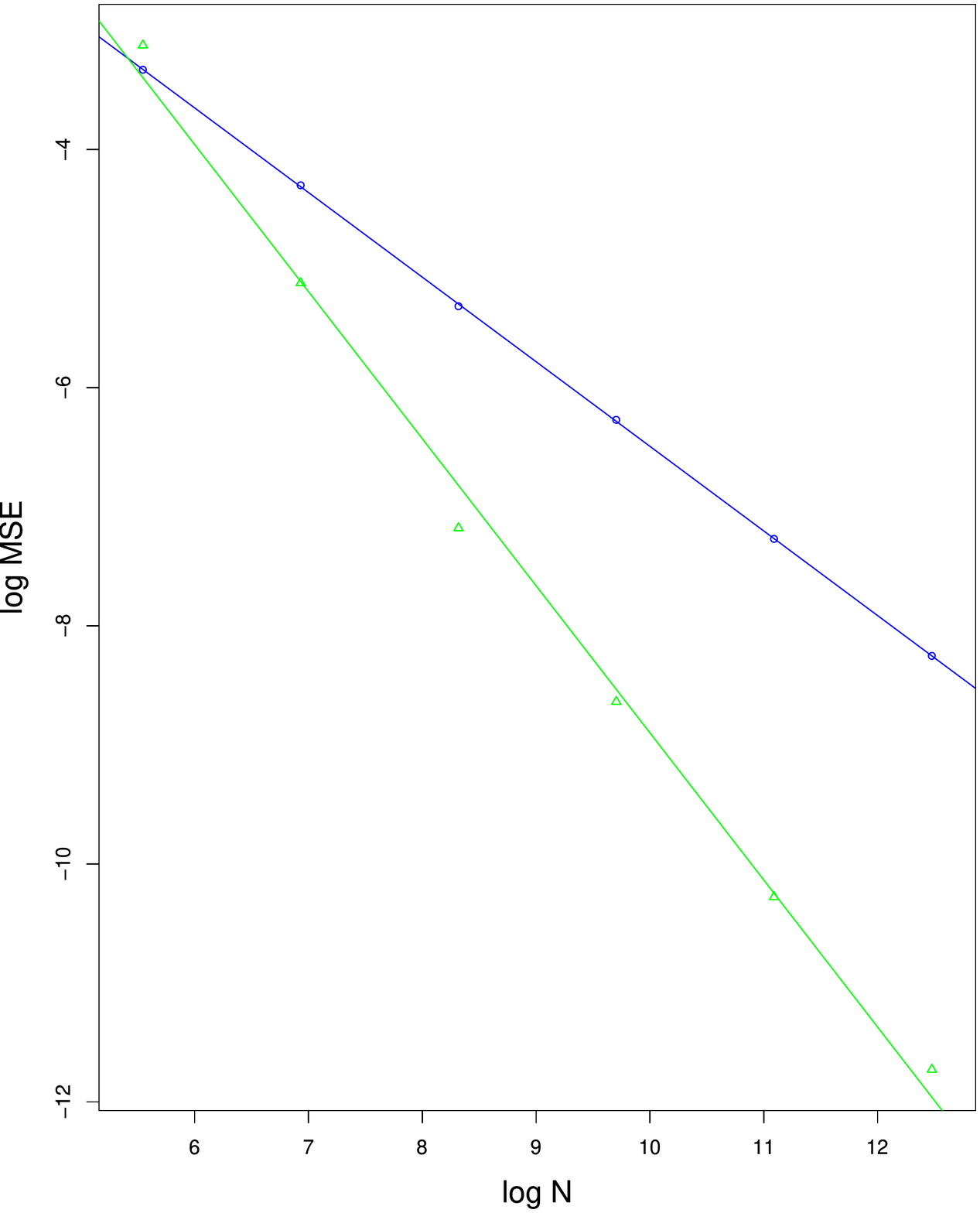}
		\includegraphics[scale=0.25]{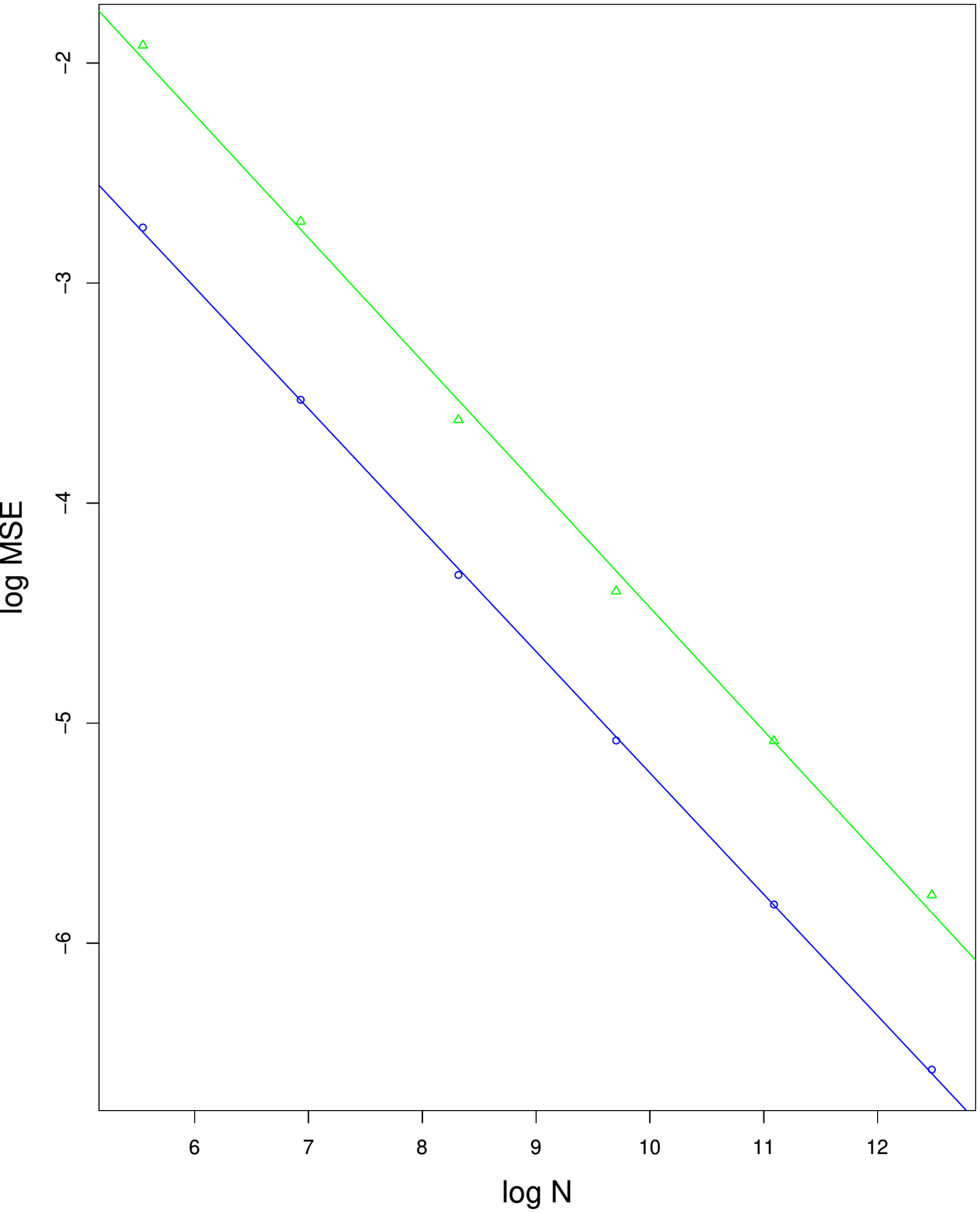} 
	\end{center}
	\caption{\textit{The two figures are $\log$ MSE vs $\log N$ plot for ideal TVD (in blue) and Dyadic CART (in green). For the figure on the left, the ground truth is piecewise constant with two pieces. The slopes came out to be $-0.71$ and $-1.23$ for ideal TVD and Dyadic CART respectively. For the figure on the right, the ground truth is a smooth bump function. The slopes came out to be $-0.55$ and $-0.56$ for ideal TVD and Dyadic CART respectively.}} 
	\label{fig4}
\end{figure}



\subsection{Simulation for univariate Dyadic Cart}
Here we compare the performance of univariate Dyadic CART of order $1$ and the constrained version of Trend Filtering (which fits piecewise linear functions) with ideal tuning. We consider a piecewise linear function $f$ given by
\begin{equation*}
	f(x) = -44 \max(0,x - 0.3) + 48 \max(0,x - 0.55) -56  \max(0,x - 0.8) + 0.28 x.
\end{equation*}
A similar function was considered in~\cite{guntuboyina2020adaptive} where the knots were at dyadic points $0.25,0.5,0.75$ respectively. We intentionally changed the knot points which makes the problem harder for Dyadic CART. We considered the ground truth $\theta^{*}$ to be evaluations of $f$ on a grid in $[0,1]$ with spacing $1/N.$ We then added standard Gaussian noise to generate data. We took a sequence of sample sizes $N$ geometrically increasing from $2^7$ to $2^{12}$ and chose $\lambda = \ell$ when $N = 2^{\ell}.$ 

\begin{figure}
	\begin{center}
		\includegraphics[scale=0.25]{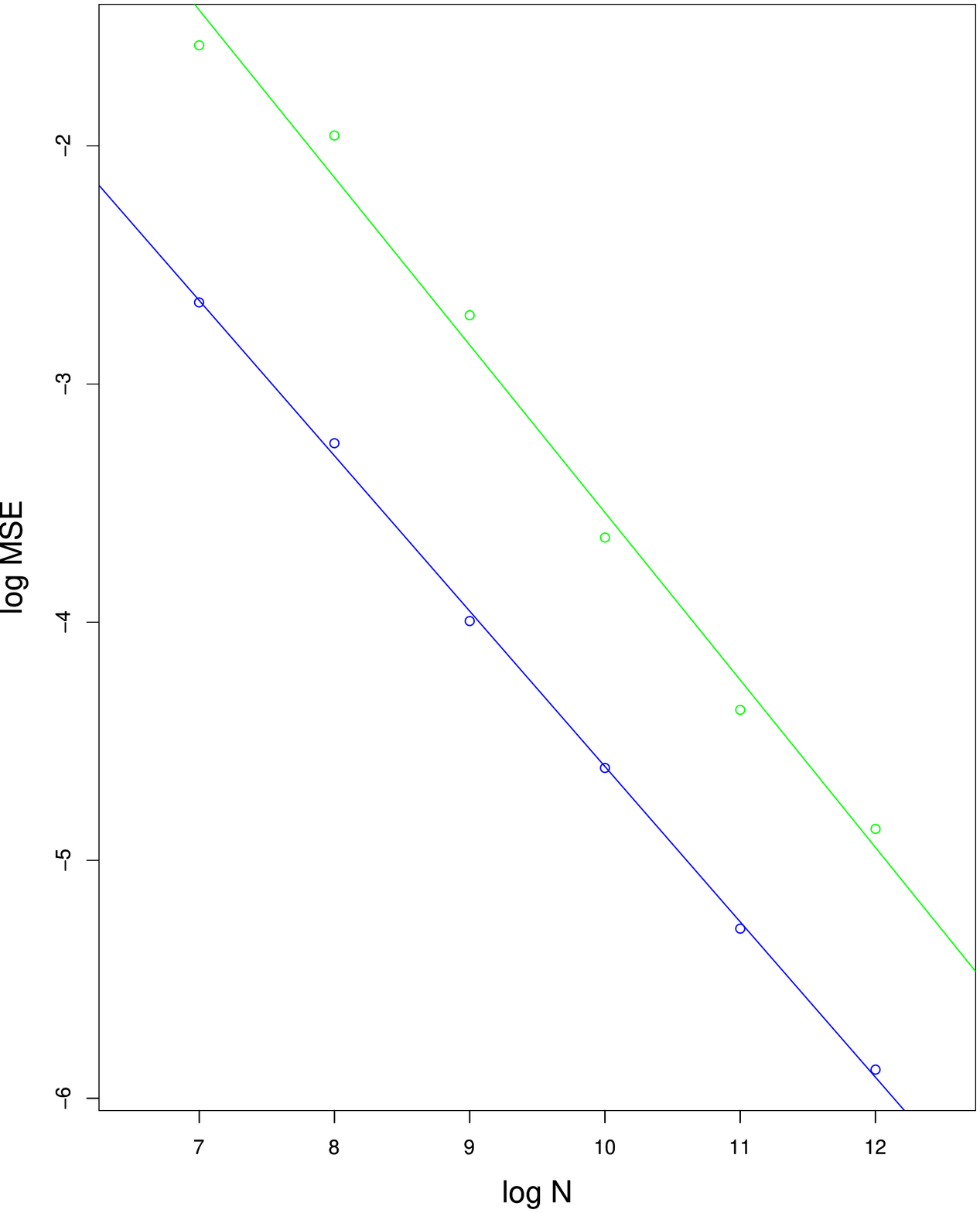}
		\includegraphics[scale=0.25]{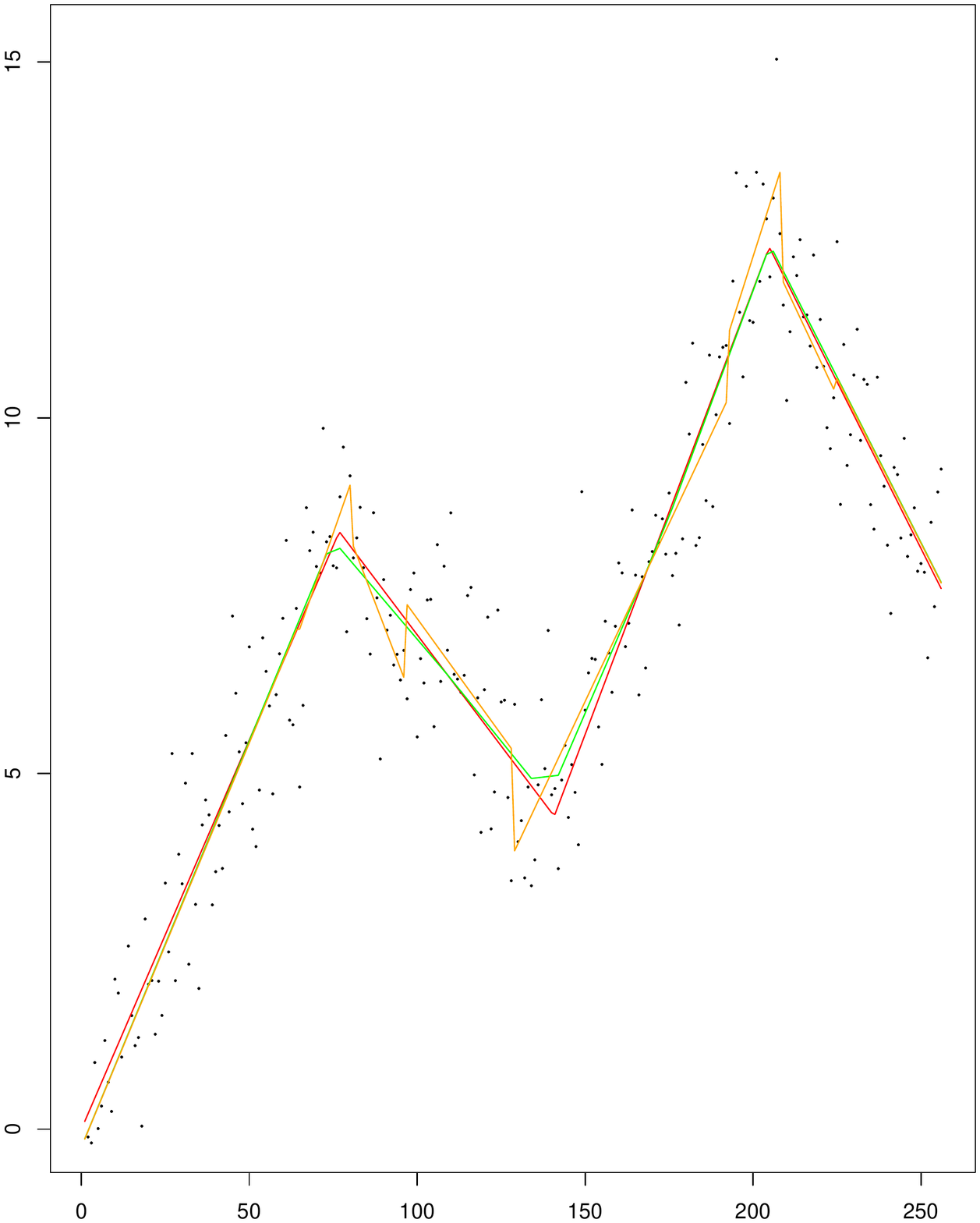} 
	\end{center}
	\caption{\textit{The figure on the left is $\log$ MSE vs $\log N$ (base $2$) plot for ideal Trend Filtering (in blue) and Dyadic CART (in green). The slopes came out to be $-0.65$ and $-0.70$ for ideal Trend Filtering and Dyadic Cart respectively. Figure on the right is an instance of our simulation with sample size $N = 256.$ The red piecewise linear curve is the ground truth. The green curve is the ideal Trend Filtering fit and the orange curve is the Dyadic CART fit with $\lambda = 8$. }} 
	\label{fig5}
\end{figure}

The slopes of the least squares lines came out to be $-0.65$ and $-0.7$ for Trend Filtering and Dyadic CART respectively (see Figure~\ref{fig5}). These slopes are a bit bigger than the theoretically expected rate $\tilde{O}(N^{-1})$ for both the estimators. This could be because of the inherent log factors. We observed that ideal Trend Filtering indeed has a lower MSE than Dyadic CART. Since the knots are at non dyadic points, fits from Dyadic CART are forced to make several knots near the true knot points. This effect is more pronounced for small sample sizes. For large sample sizes however, both the estimators give high quality fits. For a slightly worse fit than ideal Trend Filtering, the advantage of Dyadic CART is that it can be computed very fast. We have implemented Trend Filtering by using RMosek and again, we saw a significant difference in the running speeds when $N$ is large. The reason we did not take sample size larger than $2^{12}$ is again because the RMosek implementation of Trend Filtering became too slow.

\begin{remark}
	Recall that it is not necessary for the sample size to be a power of $2$ for defining and implementing Dyadic CART. We just need to adopt a convention for implementing a dyadic split of a rectangle. In our simulations, we have taken $N$ to be a power of two because writing the code becomes less messy. Also, we have compared our estimators to the ideal TVD/Trend Filtering. In practice, both estimators would be cross validated and one needs to make a comparison then as well.  
\end{remark}	


\section{Discussion}\label{Sec:discuss}
Here we discuss a couple of naturally related matters. 

\subsection{Implications for Shape Constrained Function Classes}
Since our  oracle risk bound in Theorem~\ref{thm:adapt} holds for all truths $\theta^*$, it is potentially applicable to other function classes as well. A similar oracle risk bound was used by~\cite{donoho1997cart} to demonstrate minimax rate optimality of Dyadic CART for some anisotropically smooth function classes. 
Since our focus here is on nonsmooth function classes we now discuss here some implications of our results for shape constrained function classes which has been of recent interest. 

Consider the class of bounded monotone signals on $L_{d,n}$ defined as
\begin{equation*}
	\mathcal{M}_{d,n} = \{\theta \in [0,1]^{L_{n,d}}: \theta[i_1,\dots,i_d] \leq \theta[j_1,\dots,,j_d] \:\:\text{whenever}\:\: i_1 \leq j_1,\dots,i_d \leq j_d\}.
\end{equation*}
Estimating signals within this class falls under the purview of isotonic regression. It is known that the LSE is near minimax rate optimal over $\mathcal{M}_{d,n}$ with a $\tilde{O}(n^{-1/d})$ rate of convergence for $d \geq 2$ and $O(n^{-2/3})$ rate for $d = 1$; see, e.g.~\cite{chatterjee2015risk},~\cite{chatterjee2018matrix},~\cite{han2019isotonic}. It can be checked that Theorem~\ref{thm:dcadap} actually implies that Dyadic Cart also achieves the $\tilde{O}(n^{-1/d})$ rate for signals in $\mathcal{M}_{d,n}$ as the total variation for such signals grows like $O(n^{d - 1}).$ Thus, Dyadic CART is a near minimax rate optimal estimator for multivariate Isotonic Regression as well.



Let us now consider univariate convex regression. It is known that the LSE is minimax rate optimal, attaining the $\tilde{O}(n^{-4/5})$ rate, over convex functions with bounded entries, see e.g.~\cite{GSvex},~\cite{chatterjee2016improved}. It is also known that the LSE attains the $\tilde{O}(k/n)$ rate if the true signal is piecewise linear in addition to being convex. Theorem~\ref{thm:slowrate} and Theorem~\ref{thm:fastrate} imply both these facts also hold for Univariate Dyadic Cart of order $1$. An advantage of Dyadic CART over convex regression LSE would be that Dyadic CART is computable in $O(n)$ time whereas such a fast algorithm is not known to exist yet for the convex regression LSE. Of course, the disadvantage of Dyadic CART here is the presence of a tuning parameter. 

It is an interesting question as to whether Dyadic CART or even ORT can attain near minimax optimal rates for other shape constrained classes such as multivariate convex functions etc. More generally, going beyond shape constraints, we believe that Dyadic CART of some appropriate order $r$ might be minimax rate optimal among multivariate function classes of bounded variation of higher orders. We leave this as a future avenue of research.

\subsection{Arbitrary Design}
Our estimators have been designed for the case when the design points fall on a lattice. It is practically important to have methods which can be implemented for arbitrary data. Optimal Dyadic Trees which can be implemented for arbitrary data in the context of classification have already been studied in~\cite{blanchard2007optimal},~\cite{scott2006minimax}. We now describe a similar way to define Optimal Regression Tree (ORT) fitting piecewise constant functions for arbitrary data.


Suppose we observe pairs $(x_1,y_1),\dots,(x_N,y_N)$ 
where we assume all the design points lie on the unit cube $[0, 1]^d$ after scaling, if necessary. We can divide $[0,1]^d$ into small cubes of side length $1/L$ so that there is a grid of $L^d$ cubes. We can now consider the space of all rectangular partitions where each rectangle is a union of these small cubes. Thus, each partition is necessarily a coarsening of the grid. Call this space of partitions $\mathcal{P}_{L}.$ Define $\mathcal{F}_{L}$ to be the space of functions which are piecewise constant on some partition in $\mathcal{P}_{L}.$ We can now define the optimization problem:
$$\hat{f} = \argmin_{f \in \mathcal{F}_{L}} \big(\sum_{i = 1}^{n} (y_i - f(x_i))^2 + \lambda |f|\big)$$
where $|f|$ is the number of constant pieces of $f.$


One can check that the above optimization problem can be rewritten as an optimization problem over the space of partitions in $\mathcal{P}_{L}$. The only difference with the lattice design setup is that some of the small cubes here might be empty (depending on the value $L$ and the sample size $N$) but a bottom up dynamic program can still be carried out. In this method, $L$ is an additional tuning parameter which represents the resolution at which we are willing to estimate $f^*$. Theoretical analysis need to be done to ascertain how $L$ should depend on sample size. 
The computational complexity would scale like $O(L^{2d + 1})$, so this method can still be computationally feasible for low to moderate $d.$

This method is very natural and is likely to be known to the experts but we are not aware of an exact reference. Like what is done in~\cite{bertsimas2017optimal}, it might be interesting to compare the performance of this method to the usual CART for real/simulated datasets with a few covariates. If the method performs well, theoretical analysis also needs to be done under the random design set up. We leave this for future work. 

\subsection{Dependent Errors}\label{sec:dep}
The proofs of our results can be extended without much effort to the case when $Z \sim N(0,\Sigma)$ where $\Sigma$ is a $n \times n$ covariance matrix with $\ell_2$-operator norm, i.e. the maximum eigen value $\|\Sigma\|_{{\rm op}}.$ The only change now would be that there would be an extra multiplicative factor $\|\Sigma\|_{{\rm op}}$ in front of all our risk bounds. Thus, if the maximum eigenvalue of $\Sigma$ remains bounded then all our 
rates of convergence remain the same.

In case of a general covariance matrix, the only change would be in the proof of Theorem $8.1$ which in turn relies crucially on Lemma $9.1$. In Lemma $9.1$ there are two results. One gives a bound on the expectation and the other gives a tail bound (using Gaussian concentration inequality for Lipschitz functions) for the random variable  $$\sup_{v \in S, \, v \neq \theta}  \langle Z, \frac{v - \theta}{\|v - \theta\|} \rangle$$ where $\theta$ is an arbitrary vector, $S$ is a fixed subspace of $\R^{N}$ and $Z \sim N(0,I).$ 

For a general covariance matrix $\Sigma$ we now have to give an expectation bound and a tail probability bound for a random variable of the form 
$$\sup_{v \in S, \, v \neq \theta}  \langle \Sigma^{1/2} Z, \frac{v - \theta}{\|v - \theta\|} \rangle$$ where $Z \sim N(0,1).$ Lemma~\ref{lem:0} has been written for a general positive semi definite matrix $\Sigma$ and thus gives the required expectation and tail bound. The rest of the proof is then similar and an extra multiplicative factor $\|\Sigma\|_{{\rm op}}$ will result as can be checked by the reader.

\section{Proofs}\label{Sec:proofs}
\subsection{Proof of Lemma~\ref{lem:compu}}\label{sec:algos}
\subsubsection{Case: $r = 0$}
Let us first consider the case $r = 0.$ Throughout this proof, the (multiplicative) constant involved in $O(\cdot)$ is assumed to be absolute, i.e. it does not depend on $r$ or $d$.

We describe the algorithm to compute the ORT estimator denoted by $\hat{\theta}^{(0)}_{\hier,d,n}$. Note that for any fixed $d \geq 1$ we need to compute the minimum:
$$OPT(L_{d, n}) \coloneqq \min_{\Pi: \Pi \in \mathcal{P}_{\hier,d,n}} \big(\|y - \Pi y\|^2 + K |\Pi|\big)$$
and find the optimal partition. Here $\Pi y$ denotes the orthogonal projection of $y$ onto the subspace $S^{(0)}(\Pi).$ In this case, $\Pi y$ is a piecewise constant array taking the mean value of the entries of $y_{R}$ within every rectangle $R$ constituting $\Pi.$ 

Now, for any given rectangle $R \subset L_{d, n}$ we can define the corresponding minimum restricted to $R$.
\begin{equation}
OPT(R) = \min_{\Pi: \Pi \in \mathcal{P}_{\hier,R}} \|y_{R} - \Pi y_R\|^2 + \lambda |\Pi|.
\end{equation}
where we are now optimizing only over the class of hierarchical rectangular partitions of the rectangle $R$ denoted by 
$\mathcal{P}_{\hier,R}$ and $|\Pi|$ denotes the number of 
rectangles constituting the partition $\Pi.$


A key point to note here is that due to the ``additive nature'' of the objective function over any partition (possibly trivial) of $R$ into two disjoint rectangles, we have the following {\em dynamic programming principle} for computing $OPT(R)$:
\begin{equation}
\label{eq:dynamic}
OPT(R) = \min_{R_1, R_2} (\,OPT(R_1) + OPT(R_2),\,  \|y_R - \overline{y}_R\|^2 + \lambda).
\end{equation}
Here $(R_1, R_2)$ ranges over all possible {\em nontrivial} partitions of $R$ into two disjoint rectangles. Consequently, in order to compute $OPT(R)$ and obtain the optimal heirarchical partition, the first step is to obtain the corresponding first split of $R.$ Let us denote this first split by $SPLIT(R).$

Let us now make some observations. For any rectangle $R$, the number of splits possible is at most $dn.$ This is because in each dimension, there are at most $n$ possible splits. Any split of $R$ creates two disjoint sub rectangles $R_1,R_2.$ Suppose we know $OPT(R_1)$ and $OPT(R_2)$ for $R_1,R_2$ arising out of each possible split. Then, to compute $SPLIT(R)$ we have to compute the minimum of the sum of $OPT(R_1)$ and $OPT(R_2)$ for each possible split as well as the number $1 + \|y_R - \Pi y_R\|^2$ which corresponds to not splitting $R$ at all. Thus, we need to compute the minimum of at most $nd + 1$ numbers.

The total number of distinct rectangles of $L_{d,n}$ is at most $n^{2d} = N^2.$ Any rectangle $R$ has dimensions $n_1 \times n_2 \times \dots \times n_d.$ Let us denote the number $n_1 + \dots + n_d$ by $Size(R).$ We are now ready to describe our main subroutine.

\subsubsection{Main Subroutine}
For each rectangle $R$, the goal is to store $SPLIT(R)$ and $OPT(R).$ We do this inductively on $Size(R).$ We will make a single pass/visit through all distinct rectangles $R \subset L_{d,n}$, in increasing order of $Size(R)$. Thus, we will first start with all $1\times1\times\cdots\times1$ rectangles of size equals $d$. Then we visit rectangles of size $d + 1$, $d +2$ all the way to $nd.$ Fixing the size, we can choose some arbitrary order in which we visit the rectangles.

For $1\times1\times\dots\times1$ rectangles, computing $SPLIT(R)$ and $OPT(R)$ is trivial. Consider a generic step where we are visiting some rectangle $R.$ Note that we have already computed $OPT(R^{'})$ for all rectangles $R^{'}$ with $Size(R^{'}) < Size(R).$ For a possible split of $R$, it generates two rectangles $R_1,R_2$ of strictly smaller size. Thus, it is possible to compute $OPT(R_1) + OPT(R_2)$ and store it. We do this for each possible split to get a list of at most $nd + 1$ numbers. We also compute $\|y_R - \overline{y}_R\|^2$ (described later) and add this number to the list. We now take a minimum of these numbers. This way, we obtain $OPT(R)$ and also $SPLIT(R).$

The number of basic operations needed per rectangle here is $O(nd)$. Since there are $N^2$ rectangles in all, the total computational complexity of the whole inductive scheme scales like $O(N^2\:n\:d)$.

To compute $\|y_R - \overline{y}_R\|^2$ for every rectangle $R$, we can again induct on size in 
increasing order. Define $SUM(R)$ to be the sum of entries of $y_R$ and $SUMSQ(R)$ to be the 
sum of squares of entries of $y_R.$ One can keep storing $SUM(R)$ and $SUMSQ(R)$ and $SIZE(R)$ 
bottom up. To compute these quantities for any rectangle $R$ it suffices to consider any non 
trivial partition of $R$ into two rectangles $R_1,R_2$ and then add these quantites previously 
stored for $R_1$ and $R_2.$ Therefore, this updating step requires constant number of basic 
operations per rectangle $R.$ Once we have computed $SUM(R)$ and $SUMSQ(R)$ and $SIZE(R)$ we can then calculate $\|y_R - \overline{y}_R\|^2$. Thus, this inductive sub scheme requires lower order computation.

Once we finish the above inductive scheme, we have stored $SPLIT(R)$ for every rectangle $R.$ We can now start going topdown, starting from the biggest rectangle which is $L_{d,n}$ itself. We can recreate the full optimal partition by using $SPLIT(R)$ to split the rectangles at every step. Once the full optimal partition is obtained, computing $\hat{\theta}$ just involves computing means within the rectangles constituting the partition. It can be checked that this step requires lower order computation as well.

\begin{remark}
	The same algorithm can be used to compute Dyadic CART in all dimensions as well. In this case, the number of possible splits per rectangle is $d$. Also, in the inductive scheme, we only need to visit rectangles which are reachable from $L_{d,n}$ by repeated dyadic splits. Such rectangles are necessarily of the form 
	of a product of dyadic intervals Here, dyadic intervals are interval subsets of $[n]$ which are reachable by succesive dyadic splits of $[n].$ 
	There are at most $2n$ dyadic intervals of $[n]$ and thus we need to visit at most $(2n)^d = 2^d\:N$ rectangles. 
\end{remark}

\subsubsection{Case: $r \geq 1$}
Let us fix a subspace $S \subset \R^{L_{d,n}}$ and a set of basis vectors $B = (b_1:b_2:\dots:b_L)$ for $S.$ Here, we think of $b_i$ as a column vector in $\R^N.$ For instance, $B$ may consist of all (discrete) monomials of degree at most $r$; however the algorithm works for any choice of $S$ and $B.$ We will abuse notation and also denote by $B$ the matrix obtained by stacking together the columns of $B.$ Thus $B$ is a $N \times L$ matrix; each row corresponds to an entry of the lattice $L_{d,n}.$ Also, for any rectangle $R \subset L_{d,n}$ we denote $B_R$ to be the matrix of size $|R| \times L$ where only the subset of rows corresponding to the entries in the rectangle $R$ are present. Also, let's denote the orthogonal projection matrix $B_R (B_R^T B_R)^{-1} B_R^T$ by $O_{B_R}.$

We can now again run the inductive scheme described in the last section. The important point here is that when computing $SPLIT(R)$ for each rectangle $R$, one needs to compute $y_R^{T} (I - O_{B_R}) 
y_R$. The dominating task is to compute $y_R^{T} O_{B_R} y_R.$ We can keep storing $(B_R^T B_R)$ as follows. Note that $(B_R^T B_R) = (B_{R_1}^T B_{R_1}) + (B_{R_2}^T B_{R_2})$ for any partition of $R$ 
into two subrectangles $R_1,R_2.$ Thus we can keep computing $B_R^T B_R$ inductively by adding two matrices. This is at most 
$O(L^2)$ work, i.e. requires at most $O(L^2)$ many elementary 
operations. The major step is computing the inverse $(B_R^T B_R)^{-1}$. This is at most $O(L^3)$ work. The next 
step is to compute $B_R^T y_R$. Again we can compute this by adding $(B_{R_1}^T y_{R_1}) + (B_{R_2}^T y_{R_2})$ which is $O(L)$ work. 
Finally we need to post multiply $(B_R^T B_R)^{-1}$ by $B_R^T y_R$ 
and pre multiply by $(B_R^T y_R)^T.$ This needs at most $O(L^2)$ 
work. 

\begin{remark}
	Multiplying $A_{m \times n}$ and $B_{n \times p}$ is actually $o(mnp)$ in general if we invoke Strassen's algorithm. Here we use the standard $O(mnp)$ complexity for the sake of concreteness.
\end{remark}

Per visit to a rectangle $R$, we also need to compute the minimum of $nd + 1$ numbers which is $O(nd)$ work. Finally, we need to visit $N^2$ rectangles in total. Thus the total computational complexity of ORT would be $O[N^2\:(nd + L^3)].$ 
A similar argument would give that the computational complexity of Dyadic CART of order $r \geq 1$ is $O[2^d\:N\:(d + L^3)].$ Now note that in case we use the polynomial basis, we would have $L = O(d^r).$ \qed 


\subsection{Proof of Theorem~\ref{thm:adapt}}
We will actually prove a more general result which will imply Theorem~\ref{thm:adapt}. Let $\mathcal{S}$ be any finite collection of subspaces of $\R^N.$ Recall that for a generic subspace $S\in \mathcal{S}$, we denote its dimension by $Dim(S)$ and we denote its orthogonal projection matrix by $O_{S}.$ Also, let $N_k(\mathcal{S}) = |S: Dim(S) = k, S \in \mathcal{S}|.$ Suppose, there exists a constant $c > 0$ such that for each $k \in [N],$ we have
\begin{equation}\label{eq:cardass}
N_k(\mathcal{S}) \leq N^{ck}. 
\end{equation}
Let $\Theta = \cup_{S \in \mathcal{S}} S$ be the parameter space. Let $y = \theta^* + \sigma Z$ be our observation where $\theta^* \in 
\R^n$ is the underlying mean vector and $Z \sim N(0,I).$ In this context, recall the definition of $k_{\mathcal{S}}(\theta)$ in~\eqref{eq:compl}. For a given tuning parameter $\lambda \geq 0$ we now define the usual penalized likelihood estimator $\hat{\theta}_{\lambda}$: 
\begin{equation*}
	\hat{\theta}_{\lambda} = \argmin_{\theta \in \Theta} \big(\|y - \theta\|^2 + \lambda \:k_{\mathcal{S}}(\theta)\big).
\end{equation*} 

\begin{theorem}[Union of Subspaces]\label{thm:oracleriskbd}
	Under the setting as described above, for any $0 < \delta < 1$ let us set $$\lambda \geq C \frac{\sigma^2\:\log N}{\delta}$$ for a particular absolute constant $C$ which only depends on $c$.  Then we have the following risk bound for $\hat{\theta}_{\lambda}$:
	\begin{equation*}
		\E \|\hat{\theta}_{\lambda} - \theta^*\|^2 \leq \inf_{\theta \in \Theta} \big[\frac{(1 + \delta)}{(1 - \delta)}\:\|\theta - \theta^*\|^2 + \frac{\lambda}{1 - \delta}\:k_{\mathcal{S}}(\theta)\big] + C \frac{\sigma^2}{\delta\:(1 - \delta)}.
	\end{equation*}
\end{theorem}

Using Theorem~\ref{thm:oracleriskbd}, the proof of Theorem~\ref{thm:adapt} is now straightforward. 
\begin{proof}[Proof of Theorem~\ref{thm:adapt}]
	We just have to verify the cardinality bounds~\eqref{eq:cardass} for the collection of subspaces $\mathcal{S}^{(r)}_{a}$ for $a \in \{\rdp,\hier,\all\}.$ It is enough to verify for $\mathcal{S}^{(r)}_{\all}$ because it contains the other two.
	Now $N_k(\mathcal{S}^{(r)}_{\all})$ is clearly at most the number of distinct rectangles in $L_{d,n}$ raised to the power $k.$ The number of distinct rectangles in $L_{d,n}$ is always upper bounded by $N^2.$ Thus the bound~\eqref{eq:cardass} holds with $c = 2.$ 
\end{proof}
We now give the proof of Theorem~\ref{thm:oracleriskbd}.
\begin{proof}
To avoid clutter of notations we will drop the subscript $\mathcal S$ from $k_{\mathcal 
S}(\cdot)$ in this proof. By definition, for any arbitrary $\theta \in \mathcal{S}$, we have
	\begin{equation*}
		\|y - \hat{\theta}\|^2 + \lambda \:k(\hat{\theta}) \leq \|y - \theta\|^2 + \lambda\: k(\theta).
	\end{equation*}
	Since $y = \theta^* + \sigma Z$ we can equivalently write
	\begin{equation*}
		\|\theta^* - \hat{\theta} + \sigma Z\|^2 + \lambda k(\hat{\theta}) \leq \|\theta^* - \theta + \sigma Z\|^2 + \lambda k(\theta).
	\end{equation*}
	We can further simplify the above inequality by expanding squares to obtain
	\begin{equation*}
		\|\theta^* - \hat{\theta}\|^2 \leq \|\theta^* - \theta\|^2 + \lambda k(\theta) + 2\:\langle \sigma\:Z,\hat{\theta} - \theta \rangle - \lambda k(\hat{\theta}).
	\end{equation*}
	Now using the inequality $2ab \leq \frac{2}{\delta} a^2 + \frac{\delta}{2} b^2$ for arbitrary positive numbers $a,b,\delta$ we have
	\begin{align*}
		2 \langle \sigma\:Z, \hat{\theta} - \theta \rangle &\leq \frac{2}{\delta} \big(\langle \sigma\:Z,  \frac{\hat{\theta} - \theta}{\|\hat{\theta} - \theta\|} \rangle\big)^2 + \frac{\delta}{2} \|\hat{\theta} - \theta\|^2  \\&\leq\frac{2}{\delta} \big(\langle \sigma\:Z,  \frac{\hat{\theta} - \theta}{\|\hat{\theta} - \theta\|} \rangle\big)^2 + \delta \|\hat{\theta} - \theta^*\|^2 + \delta \| \theta^* - \theta \|^2.
	\end{align*}
	The last two displays therefore let us conclude that for any $\delta > 0$ and all $\theta \in \R^N$ the following pointwise upper bound on the squared error holds: 
	\begin{equation}\label{eq:ptwiserisk}
	\|\theta^* - \hat{\theta}\|^2 \leq \frac{(1 + \delta)}{(1 - \delta)} \|\theta^* - \theta\|^2 + \frac{\lambda}{(1 - \delta)} k(\theta) + \frac{2}{\delta (1 - \delta)} \big(\langle \sigma\:Z,  \frac{\hat{\theta} - \theta}{\|\hat{\theta} - \theta\|} \rangle\big)^2 - \frac{\lambda}{(1 - \delta)} k(\hat{\theta}).
	\end{equation}
	To get a risk bound, we now need to upper bound the random variable 
	$$L(Z) = \frac{2}{\delta (1 - \delta)} \big(\langle \sigma\:Z,  \frac{\hat{\theta} - \theta}{\|\hat{\theta} - \theta\|} \rangle\big)^2 - \frac{\lambda}{(1 - \delta)} k(\hat{\theta}).$$
	
	
	For the rest of this proof, $C_1,C_2,C_3$ would denote constants whose precise value might change from line to line. 
	We can write 
	\begin{equation*}
		L(Z) \leq \max_{k \in [n]} \big[\frac{2}{\delta (1 - \delta)}\sigma^2 \sup_{S \in \mathcal{S}: Dim(S) = k} \sup_{v \in S} \big(\langle Z,  \frac{v - \theta}{\|v - \theta\|} \rangle\big)^2 \:\:-\:\:  \frac{\lambda}{(1 - \delta)} k\big].
	\end{equation*}
	
	Fix any number $t  > 0.$ Also fix a subspace $S \in \mathcal{S}$ such that $Dim(S) = k.$ Using~\eqref{eq:lem0} in Lemma~\ref{lem:0} (stated and proved in Section~\ref{Sec:appendix}) we obtain
	\begin{equation*}
		\P(\big[\frac{2}{\delta (1 - \delta)}\sigma^2 \sup_{v \in S} \big(\langle Z,  \frac{v - \theta}{\|v - \theta\|} \rangle\big)^2 \:\:-\:\:  \frac{\lambda}{(1 - \delta)} k\big] > t) \leq C_1 \exp\big(-C_2\big[\frac{t\:\delta\:(1 - \delta)}{\sigma^2} + \frac{\lambda\:k\:\delta}{\sigma^2}\big]\big). 
	\end{equation*}
	Here we also use the fact that $\lambda$ would be chosen to be at least bounded below by a constant. Using a union bound argument we can now write
	\begin{align*}
		\P(&\big[\frac{2}{\delta (1 - \delta)}\sigma^2 \sup_{S \in \mathcal{S}: Dim(S) = k} \sup_{v \in S} \big(\langle Z,  \frac{v - \theta}{\|v - \theta\|} \rangle\big)^2 \:\:-\:\:  \frac{\lambda}{(1 - \delta)} k\big] > t) \\
		&\leq N_k(\mathcal{S})\:C_1 \exp\big(-C_2\big[\frac{t\:\delta\:(1 - \delta)}{\sigma^2} + \frac{\lambda\:k\:\delta}{\sigma^2}\big]\big).
	\end{align*}
	Now use the fact that $\log N_k(\mathcal{S}) \leq c\:k\:\log N$ and set $\lambda \geq C\:\sigma^2\:\log N\:\frac{1}{\delta}$ to get a further upper bound on the right hand side of the above display
	\begin{align*}
		C_1 \exp\big(-C_2\big[\frac{t\:\delta\:(1 - \delta)}{\sigma^2} - k\:\log N\big]\big)
	\end{align*}
	The above two displays along with another union bound argument then lets us conclude 
	\begin{equation*}
		\P(L(Z) > t) \leq  \sum_{k = 1}^{n} C_1 \exp\big(-C_2\big[\frac{t\:\delta\:(1 - \delta)}{\sigma^2} - k\:\log N\big]\big).
	\end{equation*}
	Finally, integrating the above inequality with respect to all nonnegative $t$ will then give us the inequality
	\begin{equation*}
		\E L(Z) \leq C_3 \frac{\sigma^2}{\delta\:(1 - \delta)}.
	\end{equation*}
	The above inequality coupled with~\eqref{eq:ptwiserisk} finishes the proof of the proposition. \qedhere
\end{proof}

\subsection{Proof of Proposition~\ref{prop:dyadicref}}\label{sec:dyaproof}

\begin{proof}[Proof of Proposition~\ref{prop:dyadicref}]
	For simplicity of exposition, we will take $n = 2^k$ to be a power of $2$ although the same proof will go through for a general $n$. A subinterval $I$ of $[n]$ is called a {\em dyadic interval of $[n]$} if it is of the form $[(a - 1)2^{s} + 1, a2^{s}]$ for some integers 
	$0 \leq s < k$ and $1 \leq a < 2^{k - s}$. The following lemma characterizes recursive dyadic partitions in $L_{2,n}.$ 
	
	\begin{lemma}\label{lem:rdpcharac}
		A partition $\Pi \in \mathcal{P}_{\all,2,n}$ is a recursive dyadic partition iff each of its constituent rectangles is a product of dyadic intervals of $[n].$ 
	\end{lemma}

	\begin{proof}
		The only if part can be shown by an induction on the successive possible steps of constructing a recursive dyadic partition. For the if part, let us argue as follows. Let $\Pi$ be an arbitrary partition such that every rectangle constituting it is a product of dyadic intervals. If $\Pi$ is not the trivial partition $L_{2,n}$ then we will argue that there exists a dyadic split of $L_{d,n}$ into $R_1,R_2$ such that $\Pi$ is a refinement of the partition just composed of $R_1,R_2.$ Equivalently, we would show that there is a coordinate $1 \leq j \leq 2$ such that by performing a dyadic split on the $j$ th coordinate, we do not \textit{cut} any rectangle of $\Pi$ in its interior. We will now argue by contradiction. Suppose there is no such coordinate. Take coordinate $1$ for example. Then there exists a rectangle $R_1 = [a_1,b_1] \times [a_2,b_2]$ of the partition $\Pi$ which is cut in its interior by a dyadic split on the first coordinate. This necessarily implies that $a_1 < n/2$ and $b_1 > n/2.$ Since $[a_1,b_1]$ is a dyadic interval this then necessarily implies that $a_1 = 1$ and $b_1 = n$. Repeating this argument for coordinate $2$, we see that there exists rectangle $R_2$ of the partition $\Pi$ which is of the form $[a_1,b_1] \times [1,n].$ Now clearly it is not possible for any partition in $L_{2,n}$ to consist of the rectangles $R_1,R_2$ together. Thus, we have arrived at a contradiction.

		Now take a coordinate $j \in [2]$ for which a dyadic split along that coordinate does not cut any rectangle of $\Pi$ in its interior. Perform this dyadic split into $R_1,R_2.$ Now it is as if we have two separate problems of the same type within $R_1$ and $R_2.$ Again, we can find coordinates to dyadically split within $R_1$ and $R_2$ so that it does not cut any rectangle of $\Pi$ in its interior if the partition $\Pi$ within $R_1$ and $R_2$ are not trivial respectively. We can now iterate this argument till we reach the partition $\Pi$ and stop. This shows that $\Pi$ is in fact a recursive dyadic partition which finishes the proof.
	\end{proof}
	
	We would also need the following observation which is equivalent to equation~\eqref{eq:refine1} in the statement of Proposition~\ref{prop:dyadicref} for $d = 1.$

	\begin{lemma}\label{lem:1dpartbd}
		Given a partition $\Pi \in \mathcal{P}_{\all,1,n}$, there exists a refinement $\tilde{\Pi} \in \mathcal{P}_{\rdp,1,n}$ such that 
		\begin{equation*}
			|\tilde{\Pi}| \leq C |\Pi| \log_2\big(\frac{n}{|\Pi|}\big)
		\end{equation*}
		where $C > 0$ is an absolute constant.
	\end{lemma}
	
	\begin{proof}[Proof of Lemma~\ref{lem:1dpartbd}]
		Let $\Pi$ be an arbitrary partition of $[n]$ and let us denote $|\Pi|$ by $k.$ 
		Let us consider the binary tree associated with forming a RDP of $[n]$. Consider the following scheme to obtain a RDP of $[n]$ which is also a refinement of $\Pi.$ Grow the complete binary tree till the number of leaves first exceed $k.$ At this stage, each node consists of $O(n/k)$ elements. After this, if at any stage, a node of the binary tree (denoting some interval of $[n]$) is completely contained within some interval of $\Pi$, we do not split that node. Otherwise, we dyadically split the node. Due to our splitting criteria, in each such round, we split at most $k$ nodes because the nodes represent disjoint intervals. Also, the number of rounds of such splitting is at most $O(\log \frac{n}{k})$.
		When this scheme finishes, we clearly get a refinement of $\Pi$; say $\tilde{\Pi} \in \mathcal{P}_{\rdp,1,n}.$ These observations finish the proof.  
	\end{proof}

	\begin{corollary}\label{cor:simp}
		Given any interval $I \subset [n]$, there exists atmost $C \log_{2} n$ many dyadic intervals partitioning $I$.	
	\end{corollary}
	
	\begin{proof}
		Let $I = [a,b].$ Let $I_0 = [1,a - 1]$ which is empty if $a - 1 = 0$ and $I_1 = [b + 1,n]$ which is empty if $b + 1 > n.$ Then the rectangles $I_0,I,I_1$ form a partition of $[n].$ Now use Lemma~\ref{lem:1dpartbd} to obtain a recursive dyadic partition with atmost $C \log_2 n$ intervals. Considering the intervals of this recursive dyadic partition just within $I$ now finishes the proof. 
	\end{proof}
	
	Now we are ready to finish the proof of Proposition~\ref{prop:dyadicref}. Take any rectangle $R$ constituting the partition $\Pi.$ Let $R = [a_1,b_1] \times [a_2,b_2].$ Using Corollary~\ref{cor:simp}, for each $i \in [2],$ write $[a_i,b_i]$ as a union of atmost $C \log n$ many disjoint dyadic intervals. As a result, we can view $R$ itself as a union of atmost $(C\:\log n)^2$ disjoint rectangles each of which has the property that it is a product of dyadic intervals of $[n]$. Doing this for each $R$ then gives us a refinement of $\Pi$ into atmost $k (C\:\log n)^2$ rectangles each of which is a product of dyadic intervals. By Lemma~\ref{lem:rdpcharac} this refinement is guaranteed to be a recursive dyadic partition. This finishes the proof.  \end{proof}

\begin{remark}\label{rem:dyadicint}
	A natural question is whether the following holds in any dimension $d > 2$ and any $d$ dimensional array $\theta$.
	\begin{equation*}
		k^{(r)}_{\rdp}(\theta) \leq C (\:\log n)^d k^{(r)}_{\all}(\theta).
	\end{equation*}
	Proposition~\ref{prop:dyadicref} shows the above is true when $d = 1,2$. Our proof technique breaks down for higher $d.$ The reason is that Lemma~\ref{lem:rdpcharac} is no longer true when $d > 2.$ For a counter example, consider the case where $d = 3$ and $n = 2.$ Consider the partition of $L_{3,2}$ consisting of rectangles 
	\begin{enumerate}
		\item $R_1 = \{1,2\} \times \{1\} \times \{2\}$
		\item $R_2 = \{2\} \times \{1,2\} \times \{1\}$
		\item $R_3 = \{1\} \times \{2\} \times \{1,2\}$
		\item $R_4 = \{1\} \times \{1\} \times \{1\}$
		\item $R_5 = \{2\} \times \{2\} \times \{2\}$.
	\end{enumerate}
	One can check that 
	\begin{enumerate}
		\item The rectangles $R_1,\dots,R_5$ form a partition of $L_{3,2}.$
		\item Each of $R_i$ is a product of dyadic intervals of $[2].$ 
		\item This partition cannot be a recursive dyadic partition. This is because the first dyadic split itself will necessarily cut one of the rectangles $R_1,R_2,R_3$ in its interior. 
	\end{enumerate}
\end{remark}

\subsection{Proof of Theorem~\ref{thm:dcadap}}
In view of Theorem~\ref{thm:adapt}, we need to show that if $\theta \in \R^{L_{d,n}}$ has small total variation, then it can be approximated well by some $\tilde{\theta} \in \R^{L_{d,n}}$ which is piecewise constant on not too many axis aligned rectangles. To establish this, we need two intermediate results.

\begin{proposition}\label{prop:division}
	Let $\theta \in \R^{L_{d,n}}$ and $\delta > 0.$ Then there exists a Recursive Dyadic Partition $\Pi_{\theta,\delta} = (R_1,\dots,R_k) \in \mathcal{P}_{\rdp,d,n}$ such that 
	\newline
	a) $k = |\Pi_{\theta,\delta}| \leq 1 +  \log_2 N \:\:\big(1 + \frac{\TV(\theta)}{\delta}\big)$ \newline
	b) $\TV(\theta_{R_i}) \leq \delta \:\:\:\:\forall i \in [k]$ \newline
	c) $\mathcal{A}(R_i) \leq 2 \:\:\:\:\forall i \in [k]$ \newline
	where $\mathcal{A}(R)$ denotes the aspect ratio of a generic rectangle $R.$ 
\end{proposition}

\begin{proof}
In order to prove Proposition~\ref{prop:division}, we first describe a general greedy partitioning scheme --- called the $(\TV,\delta)$ scheme --- which takes as input a positive number $\delta$ and outputs a partition satisfying properties~a), b) and c). A very similar procedure was used in~\cite{chatterjee2019new}.

\medskip

\noindent {\em Description of the $(\TV,\delta)$ scheme.} First, let us note that a Recursive Dyadic Partition (RDP) of $L_{d,n}$ can be encoded via a binary tree and a labelling of the nonleaf vertices. This can be seen as follows. Let the root represent the full set $L_{d,n}.$ If the first step of partitioning is done by dividing in half along coordinate $i$ then label the root vertex by $i \in [d].$ The two children of the root now represent the subsets of $L_{d,n}$ given by $[n]^{i - 1} \times [n/2] \times [n]^{d - i}$ and its complement. Depending upon the coordinate of the next split, these vertices can now also be labelled.

In the first step of the $(\TV,\delta)$ scheme, we check whether $\TV(\theta) \leq \delta.$ If 
so, then stop and the root becomes a leaf. If not, then we label the root by $1$ and split 
$L_{d,n}$ along coordinate $1$ in half. The two vertices now represent rectangles $R_1,R_2$ 
say. For $i = 1,2$ we then check whether $\TV(\theta_{R_i}) \leq \delta.$ If so, then this 
node becomes a leaf. Otherwise, we go to the next step. In this step, we split the node along 
coordinate $2.$ We can iterate this procedure until each node in the binary tree has total 
variation at most $\delta.$ In step $i,$ we label all the nodes that are split by the number 
$i \pmod d.$ In words, we choose the the splitting coordinate from $1$ to $d$ in a cyclic 
manner. This ensures that the aspect ratio of each of the rectangles represented by the leaves is at most $2$.

After carrying out this scheme, we would be left with a Recursive Dyadic Partition of $L_{d,n}$, say, $\Pi_{\theta,\delta}$ satisfying properties~b) and c). 
In order to show that $\Pi_{\theta,\delta}$ also satisfies property~a), we need:

\begin{lemma}\label{lem:division}
	Let $\theta \in \R^{L_{d,n}}.$ Then, for any $\delta > 0$, for the $(\TV,\delta)$ division scheme, we have the following cardinality bound: 
	\begin{equation*}
		|\Pi_{\theta,\delta}| \leq 1 + \log_2 N \:\:\big(1 + \frac{\TV(\theta)}{\delta}\big)
	\end{equation*}
	Moreover, each axis aligned rectangle in $P_{\theta,\delta}$ has aspect ratio at most $2.$
\end{lemma}

\begin{proof}
	We say that a vertex of the binary tree is in generation $i$ if its graph distance to the root is $i - 1.$ Fix any positive integer $i.$ Let us consider the binary tree grown till step $i - 1.$ 
	Note that all the vertices $\{v_1,\dots,v_{k}\}$ in generation $i$ represent disjoint subsets of $L_{d,n}.$ Thus we would have $\sum_{i = 1}^{k} \TV(\theta_{v_i}) \leq V.$ This means that there can be at most $1 + \frac{\TV(\theta)}{\delta}$ vertices of generation $i$ that can be split. Now note that since there are $N$ vertices in total, the depth of the binary tree can be at most $\log_2 N.$ Thus there can be at most $\log_2 N \:\:\big(1 + \frac{\TV(\theta)}{\delta}\big)$ splits in total and each split increases the number of rectangular blocks by $1.$ This proves the cardinality bound. The second assertion is immediate from the fact that in generation $i$ the split is done on coordinate $i.$ 
\end{proof} 
This lemma together with the preceding discussions now implies our proposition.
\end{proof}

\subsubsection{Gagliardo Nirenberg Inequality}
Our next result concerns the approximation of a generic 
array $\theta$ with $\TV(\theta) \leq V$ by a constant 
matrix. This result is a crucial ingredient of the proof and is a discrete analogue of the 
Gagliardo-Nirenberg-Sobolev inequality for compactly 
supported smooth functions.

\begin{proposition}[Discrete Gagliardo-Nirenberg-Sobolev Inequality]
	\label{prop:gagliardo}
	Let $\theta \in \R^{\otimes_{i \in [d]}[n_i]}$ and 
	$$\overline \theta \coloneqq \sum_{(j_1, j_2, \ldots, j_d) \,\in\, \otimes_{i\in [d]}[n_i]}\theta[j_1, j_2, \ldots, j_d] / \prod_{i \in [d]}n_i$$ be the average of 
	the elements of $\theta$. Then for every $d > 1$ we have
	$$\sum_{(j_1, j_2, \ldots, j_d) \,\in\, \otimes_{i\in [d]}[n_i]}^m|\theta[j_1, j_2, \ldots, j_d] - \overline \theta|^{\frac{d}{d-1}} \leq \Big(1 + \max_{i, j \in [d]}\frac{n_i}{n_j}\Big)^{\frac{d}{d-1}}\TV(\theta)^{\frac{d}{d-1}}\,.$$	
As a consequence, we also have
$$\sum_{(j_1, j_2, \ldots, j_d) \,\in\, \otimes_{i\in [d]}[n_i]}^m|\theta[j_1, j_2, \ldots, j_d] - \overline \theta|^{2} \leq \Big(1 + \:\max_{i, j \in [d]}\frac{n_i}{n_j}\Big)^{2}\TV(\theta)^{2}\,.$$	
\end{proposition}

\begin{remark} Although the Gagliardo-Nirenberg-Sobolev inequality is classical for Sobolev spaces (see, e.g., Chapter 12 in~\cite{leoni2017first}), we are not aware of any discrete version in the literature that applies to arbitrary $d$ dimensional arrays. Also it is not clear if the inequality in this exact form follows directly from the classical version. 
\end{remark}

\begin{remark}
	The above Discrete Gagliardo Nirenberg inequality was already proved for the $d = 2$ case in a previous article by the authors in~\cite{chatterjee2019new}. In this article we establish this inequality for all $d \geq 2.$
	\end{remark}

\begin{proof}
	Without loss of generality we may assume that $\overline 
	\theta = 0$. Notice that, for any $(j_1, j_2, \ldots, j_d) \in \otimes_{i \in [d]} [n_i]$, we can write
	\begin{equation*}
	\label{eq:gagliardo1}
	|\theta[j_1, j_2, \ldots, j_d]|^d \leq \prod_{i \in [d]}\sum_{j_i' \in [n_i]}\big|\theta[j_1,\ldots, j_{i-1},j_i', \ldots, j_d] - \theta[j_1,\ldots, j_{i-1},j_i'-1, \ldots, j_d]\big|\,,
	\end{equation*}
	where $\theta[j_1,\ldots, j_{i-1},0, \ldots, j_d] = 0$ for all $ i \in [d]$ and $(j_1, j_2, \ldots, j_d) \in 
	\otimes_{i \in [d]}[n_i]$. For convenience we will 
	henceforth denote the difference on the right hand side 
	in the above expression as $\nabla_i\theta[j_1, \ldots, 
	j_i', \ldots, j_d]$. Now we will sum the upper bound on $|\theta[j_1, j_2, \ldots, j_d]|^{\frac{d}{d-1}}$ 
	obtained from this inequality in each of the variables 
	separately. Let us start with $j_1$:
    \begin{align*}
    \sum_{j_1 \in [n_1]}&|\theta[j_1, j_2, \ldots, j_d]|^{\frac{d}{d-1}} \nonumber\\
    &= \big(\sum_{j_1' \in [n_1]}|\nabla_1\theta[j_1', j_2, \ldots, j_d]|\big)^{\frac{1}{d-1}} \cdot \sum_{j_1 \in [n_1]}\prod_{i \in [d] \setminus \{1\}}\big(\sum_{j_i' \in [n_i]}|\nabla_i\theta[j_1, \ldots, j_i', \ldots, j_d]|\big)^{\frac{1}{d-1}}\nonumber\\
    &\leq \big(\sum_{j_1' \in [n_1]}|\nabla_1\theta[j_1', j_2, \ldots, j_d]|\big)^{\frac{1}{d-1}}\prod_{i \in [d] \setminus \{1\}} \big(\sum_{j_1 \in [n_1],\, j_i' \in [n_i]}|\nabla_i\theta[j_1, \ldots, j_i', \ldots, j_d]|\big)^{\frac{1}{d-1}}\,
    \end{align*}
	where in the final step we used the H\"{o}lder's 
	inequality. By iterating the similar steps over all the remaining $i \in [d]$, we ultimately get
	\begin{align}
	\label{eq:gagliardo3}
	\sum_{(j_1, j_2, \ldots j_d) \in \otimes_{i \in [d]} [n_i]}|\theta[j_1, j_2, \ldots, j_d]|^{\frac{d}{d-1}} \leq \prod_{i \in [d]}S_i^{\frac{1}{d-1}}\,.
	\end{align}
	where
	$$S_i \coloneqq \sum_{(j_1, j_2, \ldots j_d) \in \otimes_{i \in [d]} [n_i]}|\nabla_i\theta[j_1, \ldots, j_i, \ldots, j_d]|$$
	However, notice that
	\begin{align}
	\label{eq:gagliardo4}
    S_i  \leq  \TV(\theta) + \sum_{j_k \in [n_k],\, k \in [d] \setminus \{i\}} |\nabla_i\theta[j_1,\ldots,1,\ldots,j_d]|\,
	\end{align}
where we can recall that $\theta[j_1, \ldots, j_{i-1}, 0, \ldots, j_d] = 0$ and hence $$\nabla_i\theta[j_1,\ldots,1,\ldots,j_d] = \theta[j_1,\ldots,1,\ldots,j_d]\,.$$
Also since $\overline \theta = 0$, we can write
\begin{equation}
\label{eq:expr}
\theta[j_1,\ldots,1,\ldots,j_d] = \frac{1}{\prod_{i \in 
[d]}n_i}\sum_{(j_1', j_2', \ldots, j_d') \in \otimes_{i\in 
[d]} n_i }(\theta[j_1, \ldots, 1, \ldots, j_d] - 
\theta[j_1', j_2'\ldots, j_d'])\,.
\end{equation}
Using this expression we will show that
\begin{align}
\label{eq:gradient_bnd}
\sum_{j_k \in [n_k],\, k \in [d] \setminus \{i\}} |\nabla_i\theta[j_1,\ldots,1,\ldots,j_d] \leq  \max_{i, j \in [d]}\frac{n_i}{n_j}\,\TV(\theta)
\end{align}
which along with \eqref{eq:gagliardo3} and \eqref{eq:gagliardo4} implies the proposition. 

We will only deal with the case $i = 1$ since the other cases are similar. In the remainder of the proof we will treat the set $\mathbb L \coloneqq \otimes_{i=1}^d [n_i]$ as an (induced) subgraph of $\Z^d$. We will call the minimum length path 
between $u = (u_1, u_2, \ldots, u_d)$ and $v = (v_1, v_2, \ldots, v_d)$ in $\Z^d$ as being {\em oriented} if its first $|u_1 - v_1|$ steps are along the first 
coordinate axis, the next $|u_2 - v_2|$ steps are along the second coordinate axis and so 
on. Also for any edge $e$ in $\mathbb L$, we denote the difference between $\theta$ evaluated at its two endpoints (oriented along the direction of increasing coordinate) 
as $\nabla_e \theta$. Now writing each difference appearing in \eqref{eq:expr} as a telescoping sum of $\nabla_e\theta$ terms along an oriented path we get:
		\begin{align*}
			\label{eq:gagliardo5}
			|\nabla_1\theta&[1,j_2, \ldots,j_d]| \\
			&\leq \frac{1}{\prod_{i \in [d]}n_i}\sum_{(j_1', j_2', \ldots, j_d') \in \otimes_{i\in [d]} n_i }|\theta[1, j_2, \ldots, j_d] - \theta[j_1', j_2'\ldots, j_d']|\nonumber \\ 
			&\leq \frac{1}{\prod_{i \in [d]}n_i}\sum_{\mb j' \in \otimes_{i\in [d]} n_i}\,{ \sum_{e \in \pi_{\mb j}^{\mb j'}}|\nabla_e\theta|}\,,
		\end{align*}
	where $\pi_{\mb j}^{\mb j'}$ is the oriented (shortest-length) path between $\mb j \coloneqq (1,j_2,\ldots, j_d)$ and 
	$\mb j' \coloneqq (j_1', j_2',\ldots, j_d')$. Summing over all $\mb j \in \partial_1 \mathbb L$ (vertices in $\mathbb L$ whose first coordinate is 1) as in 
	\eqref{eq:gagliardo4} we then get
	\begin{align*}
\sum_{\mb j \in \partial_1 \mathbb L}|\nabla_i\theta[1,j_2,\ldots,j_d]| \leq \frac{1}{\prod_{i \in [d]}n_i}\sum_{e \in \mathbb L}\sum_{\mb j \in \partial_1 \mathbb L}\sum_{\mb j' \in \mathbb L}\,{ \sum_{\pi_{\mb j}^{\mb j'} \ni e}|\nabla_e\theta|}\,.
	\end{align*}
	Thus the number of times $|\nabla_e\theta|$ appears in the above summation for any given edge $e$ is at most 
	$d$ times the total number of oriented paths containing $e$ whose one endpoint lies on the boundary of $\mathbb 
	L$. Call this number $N_e$. If $e$ lies along the $i$-th coordinate axis, then it follows from the definition of oriented paths that 
	$$N_e \leq\begin{cases}
	 n_i^2\prod_{j \neq 1, i} n_j & \mbox{if } i > 1\\
	 \prod_{j \in [d]} n_j & \mbox{otherwise}\,.
	\end{cases}$$
Plugging this estimate into the previous display immediately yields \eqref{eq:gradient_bnd}.
\end{proof}

We are now ready to prove Theorem~\ref{thm:dcadap}.
\begin{proof}[Proof of Theorem~\ref{thm:dcadap}]
	Without loss of generality let $S = \{1,\dots,s\}$ where $s = |S| \geq 2.$ Let $\theta^* \in K^{S}_{d,n}(V).$ Let us denote 
	$$\mathcal{R}(\theta^*) = \inf_{\theta \in \R^{L_{d,n}}} \big[\|\theta - \theta^*\|^2 + \sigma^2 \log N \:k^{(0)}_{\rdp}(\theta)\big].$$ In view of Theorem~\ref{thm:adapt}, it suffices to upper bound $\mathcal{R}(\theta^*).$

	Let us define the $s$ dimensional array $\theta^*_{S}$ which satisfies $$\theta^*_{S}(i_1,\dots,i_s) = \theta^*(i_1,\dots,i_s,1,\dots,1) \:\:\forall (i_1,\dots,i_s) \in [n]^{s}.$$

	For any fixed $\delta > 0,$ let $\Pi_{\theta,\delta}$ be the RDP of $L_{s,n}$ that is obtained from applying Lemma~\ref{lem:division} to the $s$ dimensional array $\theta^*_{S}.$ Let $\tilde{\theta} \in \R^{L_{s,n}}$ be defined to be piecewise constant on the partition $\Pi_{\theta,\delta}$ so that within each rectangle of $\Pi_{\theta,\delta}$ the array $\tilde{\theta}$ equals the mean of the entries of $\theta^*_{S}$ inside that rectangle. Each rectangle of $\Pi_{\theta,\delta}$ has aspect ratio at most $2$ and we have
	$$k^{(0)}_{\rdp}(\tilde{\theta}) = |\Pi_{\theta,\delta}| \leq C \log N \frac{\TV(\theta^*_{S})}{\delta}.$$
	We can now apply Proposition~\ref{prop:gagliardo} to conclude, within every such rectangle $R$ of $\Pi_{\theta,\delta}$, we have $\|\tilde{\theta}_{R} - \theta^*_{R}\|^2 \leq C \delta^2.$ 
	This gives us
	\begin{equation}\label{eq:l2error}
	\|\tilde{\theta} - \theta^*_{S}\|^2 = \sum_{R \in \Pi_{\theta,\ep}} \|\tilde{\theta}_{R} - \theta^*_{R}\|^2 \leq C \delta^2 |\Pi_{\theta,\delta}| \leq C \delta \log N\: \TV(\theta^*_{S}).
	\end{equation}

	Now let us define a $d$ dimensional array $\theta^{'} \in L_{d,n}$ satisfying for any $a \in [n]^d$ the following:
	$$\theta^{'}(a) = \tilde{\theta}(a_{S}).$$ 
	Then we have $k_{\rdp}^{(0)}(\theta^{'}) = k_{\rdp}^{(0)}(\tilde{\theta}).$ We also have 
	\begin{equation}\label{eq:e1}
	\|\theta^{'} - \theta^*\|^2 = n^{d - s} \|\tilde{\theta} - \theta^*_{S}\|^2 \leq C n^{d - s} \delta \log N\: \TV(\theta^*_{S}).
	\end{equation}
	
	We can now upper bound $\mathcal{R}(\theta^*)$ by setting $\theta = \theta^{'}$ in the infimum and since $\delta > 0$ was arbitrary we can actually write 
	\begin{align}\label{eq:bdd1}
		R(\theta^*) \leq \:&C\: \inf_{\delta > 0} \big(n^{d - s} \delta \: \log N\: \TV(\theta^*_{S}) + \sigma^2 \log N \frac{\TV(\theta^*_{S})}{\delta}\big) = \\& \nonumber C\:n^{d - s} V^{*}_{S} \log N \inf_{\delta > 0} \big(\delta + \frac{\sigma^2_{S}}{\delta}\big) = C\: n^{d - s} V^*_{S} \sigma_{S} \log N.
	\end{align}
	where in the last inequality we have set $\delta = \sigma_{S}.$

	Now, let us define $\overline{\theta^*_{S}}$ as the constant $s$ dimensional array with the value being the mean of all entries of $\theta^*_{S}.$ Define a constant $d$ dimensional array $\theta^{'} \in L_{d,n}$ where every entry is again the mean of all entries of $\theta^*_{S}.$ Then we have $k^{(0)}_{\rdp}(\theta^{'}) = 1.$ In this case, we can bound $\mathcal{R}(\theta^*)$ by setting $\theta = \theta^{'}$ in the infimum to obtain
	\begin{align}\label{eq:bd2}
		\mathcal{R}(\theta^*) \leq &\:C \big(n^{d - s} \|\theta^*_{S} - \overline{\theta^*_{S}}\|^2 + \sigma^2 \log N\big) \leq Cn^{d - s}\:\big((V^*_{S})^{2} + \sigma_{S}^2\big)
	\end{align}
	where in the last inequality we have again used Proposition~\ref{prop:gagliardo}.

	Also, by setting $\theta = \theta^*$ in the infimum  and noting that $k^{(0)}_{\rdp}(\theta^*) \leq n^{s}$ we can write
	\begin{equation}\label{eq:bd3}
	\mathcal{R}(\theta^*) \leq n^s \sigma^2 \log N \leq n^{d}  \sigma_{S}^2 \log N.
	\end{equation}

	Combining the three bounds given in~\eqref{eq:bdd1},~\eqref{eq:bd2} and~\eqref{eq:bd3} finishes the proof of the theorem in the case when $s \geq 2.$

	When $s = 1$ the proof goes along exactly similar lines except there is one main difference. In place of using Proposition~\ref{prop:gagliardo} we now have to use Lemma~\ref{lem:1dapprox}, stated and proved in Section~\ref{Sec:appendix}. We leave the details of this case to be verified by the reader. \end{proof}

\subsection{Proof of Theorem~\ref{thm:tvadapmlb}}
Consider the Gaussian mean estimation problem $y = \theta^* + \sigma Z$ where $\theta^* \in K_{d,n}(V).$ Let us denote the minimax risk of this problem under squared error loss by $\mathcal{R}(V,\sigma,d,n).$ A lower bound for $\mathcal{R}(V,\sigma,d,n)$ is already known in the literature when $d \geq 2$ and is due to Theorem $2$ in~\cite{sadhanala2016total}. 
\begin{theorem}\label{thm:generalmlb}[Sadhanala et al]
	Let $V > 0,\sigma > 0$ and let $n,d$ be positive integers with $d \geq 2.$ Let $N = n^d.$ Then there exists positive universal constant $c$ such that we
	\begin{equation*}
		\mathcal{R}(V,\sigma,d,n) = \inf_{\tilde{\theta} \in \R^{L_{d,n}}} \sup_{\theta \in K_{d,n}(V)} \E_{\theta} \|\tilde{\theta} - \theta\|^2 \geq c\: \min\{\frac{\sigma\:V}{2d} \sqrt{1 + \log(\frac{2\:\sigma\:d\:N}{V})}, N \sigma^2, \frac{V^2}{d^2} + \sigma^2\}.
	\end{equation*}
\end{theorem}

We are now ready to proceed with the proof. 
\begin{proof}
	Without loss of generality, let $S = \{1,2,\dots,s\}$ where $s = |S|.$ For a generic array $\theta \in L_{d,n}$ let us define for any $a \in [n]^s$, 
	$$\overline{\theta}_{S}(a_1,\dots,a_s) = \frac{1}{n^{d - s}}\sum_{i_1,\dots,i_{d - s} \in [n]^{d - s}} \theta(a_1,\dots,a_s,i_1,\dots,i_{d - s}).$$ In words, $\overline{\theta}_{S} \in L_{n,s}$ is a $s$ dimensional array obtained by averaging $\theta$ over the $d - s$ coordinates of $S^{c}.$

	For any $\theta^* \in K^{S}_{d,n}(V)$ we can consider the reduced $s$ dimensional estimation problem where we observe $\overline{y}_{S} = \theta^*_{S} + \sigma \overline{Z}_{S}.$ This is a $s$ dimensional version of our estimation problem where the parameter space is $K_{s,n}(V_{S})$ and the noise variance is $\sigma^2_{S}.$ Hence we can denote its minimax risk under ($s$ dimensional) squared error by $\mathcal{R}(V_{S},\sigma_{S},s,n)$ 



	By sufficiency principle, $n^{d - s}$ multiplied by the minimax risk under ($s$ dimensional) squared error loss for this reduced problem is equal to the minimax risk of our original $d$ dimensional problem under the $d$ dimensional squared error loss. That is, 
	
	\begin{align*}
		\inf_{\tilde{\theta}(y) \in \R^{L_{d,n}}} \sup_{\theta \in K^{S}_{d,n}(V)} \E_{\theta} \|\tilde{\theta}(y) - \theta\|^2 &= n^{d - s} \inf_{\tilde{\theta}(\overline{y}_{S}) \in \R^{L_{s,n}}} \sup_{\theta \in K_{s,n}(V_S)} \E_{\theta} \|\tilde{\theta}(\overline{y}_{S}) - \theta\|^2 \\
		&= n^{d - s} \mathcal{R}(V_{S},\sigma_{S},s,n).
	\end{align*}

	We can now invoke Theorem~\ref{thm:generalmlb} to finish the proof when $s \geq 2.$ When $s = 1$ we note that the space $\M_{n,V} = \{\theta \in \R^n: 0 \leq \theta_1 \dots \leq \theta_n \leq V\} \subset K_{1,n}(V).$ We can now use an existing minimax lower bound for vector estimation in $\M_{n,V}$ given in Theorem $2.7$ in~\cite{chatterjee2018denoising} to finish the proof. 
\end{proof}

\subsection{Proof of Theorem~\ref{thm:slowrate}}
We first prove the following proposition about approximation of a vector in $\mathcal{BV}^{(r)}_n$ by a piecewise polynomial vector. 
\begin{proposition}\label{prop:piecewise}
	
	Fix a positive integer $r$ and $\theta \in \R^n$, and let $V^{r}(\theta) \coloneqq V.$ For any $\delta > 0$, there exists a $\theta^{'} \in \R^n$ such that \newline
	a) $k^{(r)}_{\rdp}(\theta^{'}) \leq C_r \delta^{-1/r}$ for a constant $C_r$ depending only on $r$,  and\newline
	b) $|\theta - \theta^{'}|_{\infty} \leq V \delta$
	where $|\cdot|_{\infty}$ denotes the usual $\ell_{\infty}$-norm of a vector.
\end{proposition}

\begin{remark}
	The above proposition is a discrete version of an analogous result for functions defined on the continuum in~\cite{BirmanSolomjak67}. The proof uses a recursive partitioning scheme and invokes abstract Sobolev embedding theorems which are not applicable to the discrete setting verbatim. We found that we can write a simpler proof for the discrete version which we now present. 
\end{remark}

\subsection{Proof of Proposition~\ref{prop:piecewise}}

We first need a lemma quantifying the error when approximating an arbitrary vector $\theta$ by a polynomial vector $\theta^{'}.$ This is the content of our next lemma. Recall that a vector $\theta$ is said to be a polynomial of degree $r$ if $\theta \in \mathcal{F}^{(r)}_{1,n}$ where $\mathcal{F}^{(r)}_{d,n}$ has been defined earlier. 
\begin{lemma}{\label{lem:approxpoly}}
	For any $\theta \in \R^n$ there exists a $r - 1$ degree polynomial $\theta^{'}$ such that
	\begin{equation}
	|\theta - \theta^{'}|_{\infty} \leq C_r n^{r - 1} |D^{(r)}(\theta)|_{1} = C_r V^{(r)}(\theta).
	\end{equation}
\end{lemma}


\begin{proof}
	Any vector $\alpha \in \R^n$ can be expressed in terms of $D^{(r)}(\alpha)$ and $D^{(j - 1)}(\alpha)_1$ for $j = 1,1\dots,r$ as follows:
	\begin{equation}\label{eq:buildup}
	\alpha_i = \sum_{j = 1}^{i - r} {i - j - 1 \choose r - 1} (D^{(r)}(\alpha)_j) + \sum_{j = 1}^{r} {i - 1 \choose j - 1} D^{(j - 1)}(\alpha)_1
	\end{equation}
	where the convention is that ${a \choose b} = 0$ for $b > a$, ${0 \choose 0} = 1$ and the first term in the right hand side is $0$ unless $i > r.$ This result appears as Lemma $D.2$ in~\cite{guntuboyina2020adaptive}.

	Let us define $\theta^{'}$ to be the unique $r - 1$ degree polynomial vector such that the following holds for all $j \in \{1,2,\dots,r\}$,
	\begin{equation*}
		D^{(j - 1)}(\theta^{'})_1 =  D^{(j - 1)}(\theta)_1.
	\end{equation*}
	
	Now we apply~\eqref{eq:buildup} to the vector $\theta - \theta^{'}$ to obtain 
	\begin{equation}
	(\theta - \theta^{'})_i = \sum_{j = 1}^{i - r} {i - j - 1 \choose r - 1} (D^{(r)}(\theta)_j) \leq C_r n^{r - 1} |D^{(r)}(\theta)|_{1}.
	\end{equation}
	The first equality follows from the last display and the fact that $D^{(r)}(\theta - \theta^{'}) = D^{(r)}(\theta)$ (since $D^{(r)}$ is a linear operator and $\theta^{'}$ is a $r - 1$ degree polynomial). The last inequality follows by using the simple bound ${n \choose k} \leq n^{k}$ for arbitrary positive integers $n,k.$ \end{proof}

We are now ready to proceed with the proof. 

\begin{proof}[Proof of Proposition~\ref{prop:piecewise}]
	For the sake of clean exposition, we assume $n$ is a power of $2.$ The reader can check that the proof holds for arbitrary $n$ as well. For an  interval $I \subset [n]$ let us define $$\mathcal{M}(I) = |I|^{r - 1} |D^{(r)} \theta_{I}|_1$$ where $|I|$ is the cardinality of $I$ and $\theta_{I}$ is the vector $\theta$ restricted to the indices in $I.$ Let us now perform recursive dyadic partitioning of $[n]$ according to the following rule. Starting with the root vertex $I = [n]$ we check whether $\mathcal{M}(I) \leq V \delta.$ If so, we stop and the root becomes a leaf. If not, divide the root $I$ into two equal nodes or intervals $I_1 = [n/2]$ and $I_2 = [n/2 + 1 : n].$ For $i = 1,2$ we now check whether $\mathcal{M}(I_j) \leq V \delta$ for $j = 1,2.$ If so, then this node becomes a leaf otherwise we keep partitioning. When this scheme halts, we would be left with a Recursive Dyadic Partition of $[n]$ which are constituted by disjoint intervals. Let's say there are $k$ of these intervals denoted by $B_1,\dots,B_{k}.$ 
	By construction, we have $\mathcal{M}(B_i) \leq V \delta.$ We can now apply Lemma~\ref{lem:approxpoly} to $\theta_{B_i}$ and obtain a degree $r - 1$ polynomial vector $v_i \in \R^{|B_i|}$ such that $|\theta_{B_i} - v_i|_{\infty} \leq V \delta.$ Then we can append the vectors $v_i$ to define a vector $\theta^{'} \in \R^n$ satisfying $\theta^{'}_{B_i} = v_i.$ Thus, we have $|\theta - \theta^{'}|_{\infty} \leq V \delta.$ Note that, by definition, $k^{(r - 1)}_{\rdp}(\theta^{'}) = k.$ We now need to show that $k \leq C_r \delta^{-1/r}.$

	Let us rewrite $\mathcal{M}(I) = (\frac{|I|}{n}^{r - 1}) n^{r - 1} |D^{(r)} \theta_{I}|_1.$ Note that for arbitrary disjoint intervals $I_1,I_2,\dots,I_{k}$ we have by sub-additivity of the functional $V^{r}(\theta)$,
	\begin{equation}\label{eq:subadd}
	\sum_{j \in [k]} n^{r -1} |D^{(r)} \theta_{I_j}|_1 \leq V^{r}(\theta) = V.
	\end{equation}
	The entire process of obtaining our recursive partition of $[n]$ actually happened in several rounds. In the first round, we possibly partitioned the interval $I = [n]$ which has size proportion $|I|/n = 1 = 2^{-0}.$ In the second round, we possibly partitioned intervals having size proportion $2^{-1}$. 
	In general, in the $\ell$ th round, we possibly partitioned 
	intervals having size proportion $2^{-\ell}$. Let $n_\ell$ be the number of intervals with size proportion $2^{-\ell}$ that 
	we divided in round $\ell$. Let us count and give an upper bound on $n_\ell.$ If we indeed partitioned $I$ with size proportion $2^{-\ell}$ then by construction this means 
	\begin{equation}
	n^{r - 1} |D^{(r)} \theta_{I}|_1 > \frac{V \delta}{2^{-\ell(r - 1)}}.
	\end{equation}
	Therefore, by sub-additivity as in~\eqref{eq:subadd} we can conclude that the number of such divisions is at most $\frac{2^{-\ell(r - 1)}}{\delta}.$ On the other hand, note that clearly the number of such divisions is bounded above by $2^{\ell}.$ Thus we conclude
	\begin{equation*}
		n_\ell \leq \min\{\frac{2^{-\ell(r - 1)}}{\delta},2^\ell\}.
	\end{equation*}
	Therefore, we can assert that 
	\begin{equation}
	k = 1 + \sum_{l = 0}^{\infty} n_\ell \leq  \sum_{\ell = 0}^{\infty} \min\{\frac{2^{-\ell(r - 1)}}{\delta},2^\ell\} \leq C_r \delta^{-1/r}.
	\end{equation}
	In the above, we set $n_\ell = 0$ for $\ell$ exceeding the maximum number of rounds of division possible. The last summation can be easily performed as there exists a nonnegative integer $\ell^* = O( \delta^{-1/r})$ such that 
	\begin{equation*}
		\min\{\frac{2^{-\ell(r - 1)}}{\delta},2^\ell\} = 
		\begin{cases}
			2^\ell, & \text{for} \:\:\ell < \ell^* \\
			\frac{2^{-\ell(r - 1)}}{\delta} & \text{for} \:\:\ell \geq \ell^*
		\end{cases}
	\end{equation*}
	This finishes the proof. 
\end{proof}

We can now finish the proof of Theorem~\ref{thm:slowrate}.
\begin{proof}[Proof of Theorem~\ref{thm:slowrate}]
	As in the proof of Theorem~\ref{thm:dcadap}, it suffices to upper bound
	$$\mathcal{R}(\theta^*) = \inf_{\theta \in \R^{L_{1,n}}} \big[\|\theta - \theta^*\|^2 + \sigma^2 \log n \:k^{(r)}_{\rdp}(\theta)\big].$$ By Proposition~\ref{prop:piecewise}, we obtain
	$$\mathcal{R}(\theta^*) \leq  \inf_{\delta > 0} \big[n V^2 \delta^2 + C_r \delta^{-1/r} \sigma^2 \log n\big].$$
	Setting $\delta = c (\frac{\sigma^2 V^{1/r} \log n}{n})^{2r/(2r + 1)}$ for an appropriate constant $c$ finishes the proof.
\end{proof}

\subsection{Proof of Lemma~\ref{lem:mlbpc}}\label{sec:lempc}
We will first need the following lemma. Let $\|.\|_{0}$ denote the usual $\ell_0$ norm equal to the number of non zero entries. 
\begin{lemma}\label{lem:sparse}
	Fix any positive integers $n,d$ and $1 \leq k \leq N = n^d.$ Then for any $\theta \in \R^{L_{d,n}}$, $k^{(0)}_{\hier}(\theta) \leq 3d\|\theta\|_0.$
\end{lemma}

\begin{proof}
	We will prove the lemma via induction on the dimension $d$. Let us start with the base case 
	$d = 1$. Let $1 \le i_1 < \ldots < i_{\|\theta\|_{0}} \le n$ be the support of $\theta$ (i.e. the points where $\theta$ has nonzero value). Then $\theta$ is constant on each interval of the (hierarchical) partition $\Pi \coloneqq \{\{1, \ldots, i_1-1\}, \{i_1\}, \ldots, \{i_{\|\theta_0 -1\|} + 1, \ldots, i_{\|\theta_0\|} - 1\}, \{i_{\|\theta_0\|}\}\}$ 
	of $L_{1, n}$ and consequently $k_{\hier}(\theta) \le 3\|\theta\|_0$. Now suppose the statement holds 
	for some $d \ge 1$. Let $\theta \in \R^{L_{d+1, n}}$ and $1 \le i_1 < \ldots < i_{k_0} \le 
	n$ be the horizontal coordinates of the support of $\theta$. Evidently $k_0 \le 
	\|\theta\|_0$. Now, by our induction hypothesis, for every $j \in [k_0]$, there exists a hierarchical partition $\Pi_j$ of $L_{d, n; j} \coloneqq \{i_j\} \times [n]^d$ such that $\theta$ restricted to $L_{d, n; j}$ --- which we refer to as $\theta_j$ in the sequel --- is constant on each rectangle of $\Pi_j$ and $|\Pi_j| \le 3d\|\theta_j\|_0$. Then $\theta$ 
	is constant on each rectangle of the partition formed by $\Pi_j; j \in [k_0]$ and the rectangles $\{1, \ldots, i_1 - 1\} \times [n]^{d}, \ldots, \{i_{\|k_0 -1\|} + 1, \ldots, i_{\|k_0\|} - 1\} \times [n]^{d}$. Since $\Pi$ is a hierarchical refinement of the partition comprising the rectangles $\{1, \ldots, i_1 - 1\} \times [n]^{d}, L_{d, n; 1},\ldots, 
	\{i_{\|k_0 -1\|} + 1, \ldots, i_{\|k_0\|} - 1\} \times [n]^{d}, L_{d, n; k_0}$ which is clearly hierarchical, it follows that $\Pi$ is itself a hierarchical 
	partition. Finally, notice that $$|\Pi| \le \sum_{j \in [k_0]}|\Pi_j| + 2k_0 \le \sum_{j \in [k_0]}3d\|\theta_j\|_0 + 2\|\theta\|_0 \leq 3(d + 1)\|\theta\|_0,$$
	thus concluding the induction step.
\end{proof}

Now we can finish the proof of Lemma~\ref{lem:mlbpc}.
\begin{proof}
	Recall that $\Theta_{k,d,n} \coloneqq \{\theta \in \R^{L_{d,n}}: k_{\hier}^{(0)}(\theta) \leq k\}$. Let us denote $\Theta^{({\rm sparse})}_{k} = \{\theta \in \R^{L_{d,n}}: \|\theta\|_0 \leq k\}.$ Then, by Lemma~\ref{lem:sparse} we have the set inclusion for any $1 \leq k \leq N$,
	
	$$\Theta^{({\rm sparse})}_{k} \subset \Theta_{3dk,d,n}.$$
	A minimax lower bound (for the squared error) for the parameter space $\Theta^{({\rm sparse})}_{k}$ which is tight up to constants is $C \sigma^2 k \log \frac{eN}{k}.$ This result is very well known and can be found for instance as Corollary 4.15 in~\cite{rigollet2015high}. This finishes the proof.  
\end{proof}

\section{Auxiliary Results}\label{Sec:appendix}
\begin{lemma}\label{lem:0} 
	Let $Z \sim N(0,I_n)$ and $S \subset \R^n$ denote a subspace. Let $\Sigma$ be a positive 
	semi-definite matrix and $\Sigma^{1/2}$ denote its usual square root. Also, let 
	$\|\Sigma^{1/2}\|_{{\rm op}}$ denote the $\ell_2$-operator norm of $\Sigma^{1/2}$, i.e. the 
	square root of the maximum eigen value of $\Sigma$. Then the following holds for every $\theta \in \R^n$: 
	\begin{equation*}
		\E \sup_{v \in S,\, v \neq \theta}  \langle \Sigma^{1/2} Z, \frac{v - \theta}{\|v - \theta\|} \rangle \leq \|\Sigma^{1/2}\|_{{\rm op}} (Dim(S)^{1/2} + 1).
	\end{equation*}
	Also, for any $u > 0$, we have with probability at least $1 - 2 \exp(-\frac{u^2}{2}),$
	\begin{equation}\label{eq:lem0}
	\big(\sup_{v \in S, v \neq \theta}  \langle Z, \frac{v - \theta}{\|v - \theta\|} \rangle\big)^2 \leq \|\Sigma^{1/2}\|_{{\rm op}} \big[2\:Dim(S) + 4\:(1 + u^2)\big].
	\end{equation}
\end{lemma}


\begin{proof}
	We can write
	\begin{align*}
		&\sup_{v \in S, \, v \neq \theta}  \langle  \Sigma^{1/2} Z, \frac{v - \theta}{\|v - \theta\|} \rangle = \sup_{v \in S, \, v \neq \theta} \langle \Sigma^{1/2} Z, \frac{v - O_{S} \theta - (I - O_{S})\theta}{\sqrt{|v - O_{S} \theta|^2 + |(I - O_{S})\theta|^2}} \rangle  \\\leq& \sup_{v \in S, \, v \neq \theta} \langle \Sigma^{1/2} Z, \frac{v - O_{S} \theta}{\sqrt{|v - O_{S} \theta|^2 + |(I - O_{S})\theta|^2}} \rangle + \sup_{v \in S, \, v \neq \theta} \langle \Sigma^{1/2} Z, \frac{(I - O_{S})\theta}{\sqrt{|v - O_{S} \theta|^2 + |((I - O_{S})\theta|^2}} \rangle \\\leq& \sup_{v \in S: \|v\| \leq 1}  \langle \Sigma^{1/2} Z, v \rangle + \sup_{v \in S': \|v\| \leq 1} \langle \Sigma^{1/2} Z, v\rangle
	\end{align*} 
	where $S'$ is the subspace spanned by the vector $(I - O_S)\theta$.

	Now we will give a bound on the expectation of both the terms above separately. Let us work with the first term. 
	First note that $$\sup_{v \in S: \|v\| \leq 1} \langle \Sigma^{1/2} Z,v \rangle = \sup_{v \in S: \|v\| \leq 1} \langle O_S \Sigma^{1/2} Z,v \rangle = \|O_S \Sigma^{1/2} Z\|.$$
	Secondly, 
	\begin{align*}\label{eq:gaussian}
		&\big(\E \|O_S \Sigma^{1/2} Z\|\big)^2 \leq \E \|O_S \Sigma^{1/2} Z\|^2 = \E \, Z^{t} \Sigma^{1/2} O_{S} \Sigma^{1/2} Z = \\& Trace(\Sigma^{1/2} O_{S} \Sigma^{1/2}) \leq \|\Sigma\|_{{\rm op}} Trace(O_{S}) = \|\Sigma\|_{{\rm op}} Dim(S).
	\end{align*}
	where the second last equality follows from standard facts about Gaussian quadratic forms and the last inequality follows from the following reasoning.

	Consider the spectral decomposition of $\Sigma = P D P^{t}$ where $P$ is an orthonormal matrix and $D$ is a diagonal matrix consisting of eigenvalues of $\Sigma.$ 
	Then, by commutativity of trace, we have
	\begin{align*}
		&Trace(\Sigma^{1/2} O_{S} \Sigma^{1/2}) = Trace(\Sigma O_{S}) = Trace(P D P^{t} O_{S}) =  \\& Trace(D P^{t} O_{S} P) \leq \|D\|_{{\rm op}} Trace(P^{t} O_{S} P) = \|\Sigma\|_{{\rm op}} Trace(O_{S} P P^{t}) = \|\Sigma\|_{{\rm op}} Trace (O_{S}).
	\end{align*}

	Therefore, we can say that 
	\begin{equation*}
		\E \sup_{v \in S: \|v\| \leq 1} \langle \Sigma^{1/2} Z,v \rangle \leq \|\Sigma^{1/2}\|_{{\rm op}} Dim(S)^{1/2}.
	\end{equation*}
	Similarly, for the other term we get
	\begin{equation*}
		\E \sup_{v \in S': \|v\| \leq 1} \langle  \Sigma^{1/2} Z,v \rangle  \leq \|\Sigma^{1/2}\|_{{\rm op}} Dim(S')^{1/2} \leq \|\Sigma^{1/2}\|_{{\rm op}}.
	\end{equation*}
	This proves the first part of the lemma.

	Coming to the second part, note that, by symmetry we also have the lower bound
	$$\E \sup_{v \in S, v \neq \theta}  \langle  \Sigma^{1/2} Z, \frac{v - \theta}{\|v - \theta\|} \rangle \ge \E \inf_{v \in S, v \neq \theta}  \langle  \Sigma^{1/2} Z, \frac{v - \theta}{\|v - \theta\|} \rangle \geq - \|\Sigma^{1/2}\|_{{\rm op}} \big(Dim(S)^{1/2} + 1\big).$$	
	Now we note that $\sup_{v \in S,\, v \neq \theta}  \langle \Sigma^{1/2} Z, \frac{v - \theta}{\|v - \theta\|} \rangle$ is a Lipschitz function of $Z$ with Lipschitz constant $\|\Sigma^{1/2}\|_{{\rm op}}$. Thus we can use the well-known Gaussian Concentration inequality (see, e.g. \cite[Theorem~7.1]{L01}), which is stated as Theorem~\ref{prop:gauss_conc} for the convenience of the reader, to conclude that for any $u > 0,$ with probability at least $1 - 2 \exp(-u^2/2)$ we have 
	\begin{equation*}
		|\sup_{v \in S}  \langle Z, \frac{v - \theta}{\|v - \theta\|} \rangle| \leq \|\Sigma^{1/2}\|_{{\rm op}} \,|\sqrt{Dim(S)} + 1 + u|
	\end{equation*}
	Using the elementary inequality $(a + b)^2 \leq 2 a^2 + 2 b^2$ now finishes the proof of \eqref{eq:lem0}. 
\end{proof}



\begin{theorem}\label{prop:gauss_conc}
	Let $Z_1, Z_2,\dots, Z_m$ be independent standard Gaussian variables and $f: \R^m \mapsto \R$ be a Lipschitz function with Lipschitz constant $1$. Then $\E f(Z_1, \dots, Z_m)$ is finite and
	$$ P (|f(Z_1, \dots, Z_m) - \E f(Z_1, \dots, Z_m)| > t) \leq 2 \exp^{-t^2/2}$$
	for all $t\geq 0$.
\end{theorem}

The following lemma appears as Lemma~7.3 in~\cite{chatterjee2019new}.
\begin{lemma}\label{lem:1dapprox}
	Let $\theta \in \R^n.$ Let us define $\overline{\theta} = (\sum_{i = 1}^{n} \theta_i)/n.$ Then we have the following inequality:
	\begin{equation*}
		\sum_{i = 1}^{n} \big(\theta_i - \overline{\theta}\big)^2 \leq n \TV(\theta)^2\,.
	\end{equation*}
\end{lemma}

\bibliographystyle{chicago}
\def\noopsort#1{}

\end{document}